\documentclass[11pt,reqno]{amsart}

\numberwithin{equation}{section}

\usepackage{url}
\usepackage[breaklinks]{hyperref}
\usepackage[T1]{fontenc}
\usepackage{color} 
\definecolor{darkred}{rgb}{1,0,0} 
\definecolor{darkgreen}{rgb}{0,0.8,0}
\definecolor{darkblue}{rgb}{0,0,1}

\hypersetup{colorlinks, linkcolor=darkblue, filecolor=darkgreen,
  urlcolor=darkred, citecolor=darkgreen}

\usepackage{latexsym}
\usepackage{a4wide}
\usepackage{amscd}

\usepackage{graphicx}

\usepackage{marginnote}

\makeatletter
\DeclareFontFamily{OMX}{MnSymbolE}{}
\DeclareSymbolFont{MnLargeSymbols}{OMX}{MnSymbolE}{m}{n}
\SetSymbolFont{MnLargeSymbols}{bold}{OMX}{MnSymbolE}{b}{n}
\DeclareFontShape{OMX}{MnSymbolE}{m}{n}{
    <-6>  MnSymbolE5
   <6-7>  MnSymbolE6
   <7-8>  MnSymbolE7
   <8-9>  MnSymbolE8
   <9-10> MnSymbolE9
  <10-12> MnSymbolE10
  <12->   MnSymbolE12
}{}
\DeclareFontShape{OMX}{MnSymbolE}{b}{n}{
    <-6>  MnSymbolE-Bold5
   <6-7>  MnSymbolE-Bold6
   <7-8>  MnSymbolE-Bold7
   <8-9>  MnSymbolE-Bold8
   <9-10> MnSymbolE-Bold9
  <10-12> MnSymbolE-Bold10
  <12->   MnSymbolE-Bold12
}{}

\let\llangle\@undefined
\let\rrangle\@undefined
\DeclareMathDelimiter{\llangle}{\mathopen}%
                     {MnLargeSymbols}{'164}{MnLargeSymbols}{'164}
\DeclareMathDelimiter{\rrangle}{\mathclose}%
                     {MnLargeSymbols}{'171}{MnLargeSymbols}{'171}

\newcommand{\N}{\mathbb{N}}                     
\newcommand{\Z}{\mathbb{Z}}                     
\newcommand{\R}{\mathbb{R}}                     
\newcommand{\C}{\mathbb{C}}                     
\newcommand{\E}{\mathbb{E}}                    
\newcommand{\id}{\mathrm{id}}       
\newcommand{\crit}{\mathrm{Crit\,}}               
\newcommand{\Cont}{\mathrm{Cont}}     

\theoremstyle{plain}
\newtheorem{thm}{Theorem}[section]
\newtheorem{lem}[thm]{Lemma}

\newtheorem{cor}[thm]{Corollary}

\theoremstyle{definition}
\newtheorem{defn}[thm]{Definition}   
\newtheorem{rem}[thm]{Remark}

\title{Positive loops and $L^{\infty}$-contact systolic inequalities}
\author{Peter Albers}
\author{Urs Fuchs}
\author{Will J. Merry}

\address{ Peter Albers\\
 Mathematisches Institut\\
 Ruprecht-Karls-Universit\"at Heidelberg}
\email{palbers@{}mathi.uni-heidelberg.de}
\address{ Urs Fuchs\\
 Mathematisches Institut\\
 Ruprecht-Karls-Universit\"at Heidelberg}
\email{ufuchs@mathi.uni-heidelberg.de}
\address{Will J.~Merry\\
 Department of Mathematics\\
 ETH Z\"urich}
\email{merry@math.ethz.ch}

\begin{document}

\begin{abstract}
We prove an  inequality between the $L^{\infty}$-norm of the contact Hamiltonian of a positive loop of contactomorphims and the minimal Reeb period. This implies that there are no small positive loops on hypertight or Liouville fillable contact manifolds. Non-existence of small positive loops for overtwisted 3-manifolds was proved by Casals-Presas-Sandon in \cite{CasalsPresasSandon2016}.

As corollaries of the inequality we deduce various results. E.g.~we prove that certain periodic Reeb flows are the \emph{unique} minimizers of the $L^{\infty}$-norm. Moreover, we establish $L^\infty$-type contact systolic inequalities in the presence of a positive loop.
\end{abstract}

\maketitle{}

\section{Introduction}

In 2000 Eliashberg and Polterovich \cite{EliashbergPolterovich2000} introduced the notion of positivity in contact geometry and dynamics. Jointly with Kim in 2006 \cite{EliashbergKimPolterovich2006} they related positivity with contact (non-)squeezing and rigidity in contact dynamics. The central idea is that of a positive path of contactomorphisms, which we now explain.

A contact manifold $(\Sigma, \xi)$ consists of an $(2n-1)$-dimensional manifold $\Sigma$ equipped with a maximally non-integrable hyperplane distribution $\xi \subset T \Sigma$. Throughout this article we always assume $\xi$ is coorientable, which means that there exists a one-form $\alpha \in \Omega^1(\Sigma)$ such that $\xi = \ker \alpha$. The maximal non-integrability condition can then be conveniently expressed as the assertion that $\alpha \wedge d \alpha^{n-1}$ is nowhere vanishing. Such an one-form $\alpha$ is called a contact form. It determines a vector field $R_{\alpha}$ called the Reeb vector field by requiring that $i_{R_{\alpha}} d \alpha = 0$ and $\alpha(R_{\alpha}) = 1$. We fix a coorientation of $\xi$ and consider only contact forms such that $R_{\alpha}$ is positive with respect to this coorientation; we denote the set of such $\alpha$ by $\mathcal{C}(\Sigma, \xi)$.

The identity component $\Cont_0(\Sigma, \xi)$ of the group of contactomorphisms consists of those diffeomorphisms $\psi : \Sigma \to \Sigma$ which satisfy  $\psi_* \xi = \xi$ and are isotopic to the identity through such diffeomorphisms. The condition $\psi_* \xi = \xi$ can be rephrased as saying that $\psi^* \alpha = \rho \alpha$ for a positive function $\rho : \Sigma \to (0,\infty)$ which is called the conformal factor of $\psi$. In this article a path $\varphi = \{ \varphi_t\}_{t \in [0,1]}$ of contactomorphisms always starts at the identity: $\varphi_0 = \id$. A path of contactomorphisms $\varphi$ is generated by a time-dependent vector field $X_t$, i.e.~$X_t \circ \varphi_t = \tfrac{d}{dt} \varphi_t$. 
The function $h^{\alpha} : \Sigma \times [0,1] \to \R$ defined by
\begin{equation}
  h^{\alpha}_t \big(\varphi_t(x) \big) := \alpha_x\big(X_t(x)\big) 
\end{equation}
is called the contact Hamiltonian of the path $\varphi$. A path $\varphi$ is called \emph{positive} (resp.~\emph{non-negative}) if $h^\alpha$ is positive (resp. non-negative) and \emph{non-trivial} if $h^\alpha\neq 0$. 
These conditions are independent on the contact form $\alpha$ defining the contact structure; for instance the positivity condition means just that $X_t$ is positively transverse to $\xi$ for each $t\in [0,1]$. In the following, a positive path means a positive path of contactomorphisms. The flow $\theta_t^{\alpha}$ of the Reeb vector field $R_{\alpha}$ has contact Hamiltonian the constant function $1$, and thus is always a positive path. 

As explained in \cite{EliashbergPolterovich2000,EliashbergKimPolterovich2006}, a particularly important role is played by \emph{positive loops} of contactomorphisms $\varphi = \{ \varphi_t \}_{t \in S^1}$. In this article we prove a sharp inequality between the $L^{\infty}$-norm of the contact Hamiltonian of a positive loop of contactomorphims and the minimal Reeb period and deduce  $L^{\infty}$-contact systolic inequalities in the presence of a positive loop. 

From now on we will only consider closed connected contact manifolds. Given a free homotopy class $\nu \in [S^1, \Sigma]$ and $\alpha \in \mathcal{C}(\Sigma, \xi)$, denote by 
\begin{equation}
\label{eq:min_period}
  \wp(\alpha ,\nu) := \inf \Big\{ T>0 \mid 
   \exists \text{ a closed orbit } \gamma \text{ of } R_\alpha \text{ in the class } \nu \text{ with period }T\Big\} .
\end{equation}
This number is either attained and thus positive or infinite,
\begin{equation}
  \wp (\alpha , \nu) >0 ,\qquad \forall \,\nu \in [S^1, \Sigma],
 \end{equation} 
 since $R_{\alpha}$ is a non-vanishing vector field. We denote the component of contractible loops by $\nu_0$. Given a loop of contactomorphisms $\varphi = \{ \varphi_t \}_{t \in S^1}$ of $\Sigma$, we denote by $\nu_{\varphi}$ the free homotopy class containing the loops $t \mapsto \varphi_t(x)$ for some $x \in \Sigma$.
Given a path $\varphi_t$ of contactomorphisms, we denote by $\rho^{\alpha}_t : \Sigma \to (0,\infty)$ the associated family of conformal factors, and we set 
\begin{equation}
\label{eq:pi_alpha_phi}
   \varrho^{\alpha}(\varphi) := \min_{(x,t ) \in \Sigma \times [0,1]} \rho^{\alpha}_t(x)  \in (0,1].
 \end{equation} 
Here the positivity of $\varrho^\alpha(\varphi)$ is a consequence of the compactness of $\Sigma$ and $\varrho^{\alpha} (\varphi)\le 1$ follows from the fact that by assumption $\varphi_0 = \mathrm{id}$ and hence $\rho^{\alpha}_0 \equiv 1$.

\begin{rem}
Throughout this article we make the convention that a Reeb orbit of period $\eta\in\R$ is a curve $\gamma:S^1\to \Sigma$ satisfying the equation
\begin{equation}
\dot\gamma=\eta R_\alpha(\gamma)\;.
\end{equation}
In particular, we allow constant curves in case $\eta=0$.
\end{rem}

\begin{thm}
\label{thm:thm1_intro}
Fix $\alpha \in \mathcal{C}(\Sigma,\xi)$ and let $\varphi$ be a loop of contactomorphisms with contact Hamiltonian $h^{\alpha}$. If 
\begin{equation}
0<\| h^{\alpha} \|_{\mathrm{osc}} < \varrho^{\alpha}(\varphi) \cdot \wp(\alpha,  \nu_0) 
\end{equation}
then there is a closed Reeb orbit belonging to the free homotopy class $\nu_{\varphi}$ with period $\eta$ such that 
\begin{equation}
\|h^{\alpha}\|_-< \eta \leq \|h^{\alpha}\|_+.
\end{equation}
As a consequence, if $\varphi$ is a non-trivial non-negative loop with $\| h^\alpha \|_{\mathrm{osc}} <\varrho^\alpha(\varphi) \cdot \wp(\alpha,  \nu_0) $, then
\begin{equation}
  \wp(  \alpha, \nu_{ \varphi}) \le \| h^{\alpha} \|_+.
\end{equation}
\end{thm}
Here we used the following notation: for a function $h : \Sigma \times [0,1] \to \R$,
\begin{equation}
  \| h \|_+ := \int_0^1 \max_{x \in \Sigma } h_t(x) \,dt, \qquad \| h\|_- := \int_0^1 \min_{x \in \Sigma} h_t(x)\,dt,
\end{equation}
and
\begin{equation}
  \| h \|_{\mathrm{osc}} := \| h\|_+ - \|h\|_-.
\end{equation}
Thus when $h$ is a non-negative function, one has $0 \le \| h \|_- \le \| h \|_+$. Moreover $\| h\|_{ \mathrm{ osc}} = 0$ if and only if $h_t$ is constant for each $t\in [0,1]$; 
thus if a path $\varphi_t$ has contact Hamiltonian $h^{\alpha}$ with $\| h^{\alpha} \|_{\mathrm{osc}} = 0$ 
then $\varphi_t = \theta^{\alpha}_{c(t)}$ for some $c:[0,1]\rightarrow \R$.

\begin{rem}
\label{cor:main_theorem_1} 
Theorem \ref{thm:thm1_intro} implies in particular that we have for any positive loop $\varphi$

\begin{equation}
\label{eq:dichotomy}
    \frac{\| h^{\alpha} \|_{\mathrm{osc}}}{\varrho^{\alpha}(\varphi)} \ge \wp(\alpha , \nu_0) \quad \text{ or } 
    \quad \| h^\alpha\|_+ \ge \wp( \alpha , \nu_{\varphi})
\end{equation}
A previous result in this direction is due to Casals-Sandon-Presas \cite{CasalsPresasSandon2016}. 
They proved that on overtwisted manifold there exists a (non-explicit) constant $C(\alpha)$ such that any 
positive loop $\varphi$ satisfies $\| h^{\alpha}\|_+ \ge C(\alpha) $. 
Moreover a proof of the fact that on any compact contact manifold it is impossible to have a 
positive contractible loop $\varphi$ which is generated by a contact Hamiltonian $h^{\alpha}$ with 
$ \| h^{\alpha} \|_+ < \wp(\alpha, \nu_0)$ has been recently announced by Sandon and will appear in \cite{Sandon201?}.
\end{rem}

\begin{rem}
Let us briefly explain the assumptions and the general scheme of the proof of Theorem \ref{thm:thm1_intro}. The Reeb orbit we find in the Theorem is obtained by studying a Rabinowitz action functional perturbed by the loop of contactomorphisms $\varphi$. Indeed, finding a critical point of this perturbed functional automatically produces a periodic Reeb orbit, see Figure \ref{pic:figure8}. The existence of critical points follows from a stretching argument. In order to control compactness during this stretching procedure we need the assumptions. More precisely, the norm of the contact Hamiltonian is relevant for estimating the energy of gradient trajectories. Since we are working in a symplectization the quantity $ \varrho^{\alpha}(\varphi)$ enters in order to apply SFT-compactness. This last step is simplified if we add the assumption of hypertightness or Liouville fillability, see Corollary \ref{thm:hypertight_small} and Theorem \ref{thm:thm1_intro_Liouville_fillable}.
\end{rem}

An interesting subclass of positive loops are contractible positive loops, since they can be used for contact squeezing, see \cite{EliashbergKimPolterovich2006}. Equation \eqref{eq:dichotomy} immediately implies that it is not possible to contract such a positive loop through positive loops to the trivial loop. For the standard contact sphere \cite{EliashbergKimPolterovich2006} proved a stronger result, namely the contact Hamiltonian of any contracting homotopy must become negative by a definite amount at some point. 

Examples of contact manifolds come from unit cotangent bundles $S^*B$ for Riemannian manifolds $(B,g)$. In this case the Reeb flow is the geodesic flow of $g$. For instance, if $(B,g)$ is $S^2$ with the round metric, the Reeb flow is a positive loop, since all geodesics are closed. Unit cotangent bundles rarely admit positive loops \cite{AlbersFrauenfelder2012,FrauenfelderLabrousseSchlenk2015}. However there are many classes of contact manifolds which do, the most ubiquitous of which is prequantisation spaces, see below. Moreover, many contact 3-manifolds also carry positive loops \cite{CasalsPresas2016}.

The estimates in \eqref{eq:dichotomy} is sharp in the following examples. The simplest is the sphere $S^{2n-1}$ equipped with its standard contact form $\alpha$. The Reeb flow is periodic, and we normalise the contact form $\alpha$ by requiring that the Reeb flow is one-periodic, and hence is a positive loop. Then by definition $\wp( \alpha, \nu_0) = 1$. Both the contact Hamiltonian and the conformal factor of the Reeb flow is the constant function 1. Thus the oscillation norm of the contact Hamiltonian is zero, and we have equality in the second of the two options in  \eqref{eq:dichotomy}. More generally consider a lens space $L(1,p) = S^{2n-1} / \Z_p$. Then  the rescaled contact form $p \alpha$ on $S^{2n-1}$ descends to a contact form $\bar{\alpha}$ on $L(1,p)$ whose Reeb flow  $\bar{\theta}$ is again  one-periodic. Now we have $\wp( \bar{\alpha}, \nu_0) = p$ and $\wp( \bar{\alpha}, \nu_{\bar{\theta}}) = 1$.

This is a special case of a phenomenon on more general  prequantisation spaces.  Here one begins with a closed symplectic manifold $(M,\omega)$ for which the de Rham cohomology class
$[\omega]$ has an integral lift in $\mbox{H}^{2}(M;\mathbb{Z})$.
Consider a circle bundle $p:\Sigma_k \rightarrow M$ with Euler class
$k[\omega]$ for some $k \in \Z$ with $k \ne 0$, and connection 1-form $\alpha_k$ with $p^{*}(k\omega)=-d\alpha_k$.
Then $(\Sigma_k,\alpha_k)$ is a contact manifold whose associated Reeb
flow is one-periodic (since we think of $S^1$ as $\R / \Z$). The closed Reeb orbits are the fibres of this bundle.
The long exact homotopy sequence of the fibration is
\begin{equation}
\pi_{2}(M)\overset{q_k}{\rightarrow}\pi_{1}(S^{1})\rightarrow\pi_{1}(\Sigma_k)\rightarrow\pi_{1}(M)\rightarrow0.
\end{equation}
The map $q_k$ is non-trivial if and only if the homotopy class of the fibre is torsion. If we identify $\pi_1(S^1) \cong \Z$ then the map $q_k$ agrees with $I_{\omega} : \pi_2(M) \to \Z$, where $I_{\omega}$ is given via Chern-Weil theory by integrating against $\omega$. In particular $\wp(\alpha_k, \nu_{ \theta_k}) = 1 $ and $\wp( \alpha_k , \nu_0)$ is equal to $+ \infty$ if the homotopy class of the  fibre is not torsion, i.e. if $I_{\omega}= 0$. Otherwise there exists $N_{\omega} \in \N$ such that $I_{\omega}(\pi_2(M)) = N_{\omega} \Z$ and $\wp( \alpha_k, \nu_0) = |k| N_{\omega}$.

In general there is no lower bound of the form $\|h^{\alpha}\|_-\leq \wp(\alpha, \nu_{\varphi})$ 
for a positive loop, even if it is of small oscillation. 
Indeed, consider a one-periodic Reeb flow $\theta^{\alpha}_t$ with contractible orbits. 
Then if $\varphi_t := \theta^{\alpha}_{2t}$ then $h^{\alpha} = 2$ and hence $ 2 = \| h^{\alpha} \|_- > 
\wp(\alpha, \nu_{\varphi}) = 1$.

\subsection*{Hypertight contact manifolds}
A contact manifold $(\Sigma, \xi)$ is called \emph{hypertight} if there exists $\alpha \in \mathcal{C}(\Sigma, \xi)$ with no contractible Reeb orbits. See for example \cite{ColinHonda2005} for a construction
of hypertight contact manifolds. The 3-torus $\mathbb{T}^{3}$ (equipped with
any one of the standard contact structures $\alpha_{k}=\cos(2\pi kr)ds+\sin(2\pi kr)dt$)
is a familiar example. Prequantisation spaces are hypertight if and only if if and only if the fibre is not torsion, as proved in our previous paper \cite{AlbersFuchsMerry2015}. In this case $\wp( \alpha ,\nu_0) = \infty$, and hence Theorem \ref{thm:thm1_intro} is applicable to \emph{all} positive loops. This yields the following statement, which is a quantatitive version of the main result of \cite{AlbersFuchsMerry2015}. 
\begin{cor}  \label{thm:hypertight_small} 
Suppose $(\Sigma ,\alpha)$ is a contact manifold such that $\alpha$ has no contractible Reeb orbits. Then given any loop of contactomorphisms $\varphi$ with $\|h^\alpha\|_{\mathrm{osc}}>0$, there exists a closed Reeb orbit belonging to $\nu_{\varphi}$ with period  $\eta$ such that 
\begin{equation}
\|h^{\alpha}\|_-< \eta \leq \|h^{\alpha}\|_+
\end{equation}
As a consequence, if $\varphi$ is a non-trivial non-negative loop, then
\begin{equation}
\|h^{\alpha}\|_+\geq \wp(\alpha,\nu_\varphi)
\end{equation} 
and thus on hypertight contact manifolds there are no $C^0$-small positive loops.
\end{cor}

\subsection*{Fillable contact manifolds}

A contact manifold $(\Sigma,\xi)$ is \emph{Liouville fillable} if there exists a compact exact symplectic manifold $(W,\omega=d\lambda)$ with $\partial W=\Sigma$ such that if  $\alpha :=\lambda|_\Sigma$ then $\alpha \in \mathcal{C}(\Sigma,\xi)$. Moreover, we require the Liouville vector field associated with $\lambda$ points outward along $\Sigma$. Examples are $\Sigma=S^{2n+1}$ and unit cotangent bundles $S^*B$.

\begin{rem}
Everywhere where we treat Liouville fillable contact manifolds we redefine the free homotopy class  $\nu_0$ to be the free homotopy class of the constant loop in $[S^1,W]$. With this redefinition $ \wp(\alpha,  \nu_0) $ is still positive but becomes potentially smaller.
\end{rem}

\begin{thm}
\label{thm:thm1_intro_Liouville_fillable}
Suppose $(\Sigma ,\xi)$ is Liouville fillable and $\alpha\in \mathcal{C}(\Sigma,\xi)$. Let $\varphi$ be a loop of contactomorphisms with contact Hamiltonian $h^{\alpha}$. If 
\begin{equation}\label{eqn:thm1_intro_Liouville_fillable}
0<\| h^{\alpha} \|_{\mathrm{osc}} < \wp(\alpha,  \nu_0) 
\end{equation}
then there is a closed Reeb orbit belonging to the free homotopy class $\nu_{\varphi}$ with period $\eta$ such that 
$\|h^{\alpha}\|_-< \eta \leq \|h^{\alpha}\|_+$.
As a consequence, if $\varphi$ is a non-trivial non-negative loop with $\| h^{\alpha} \|_{\mathrm{osc}} <  \wp(\alpha,  \nu_0)$, then
\begin{equation}
  \wp(  \alpha, \nu_{ \varphi}) \le \| h^{\alpha} \|_+.
\end{equation}
\end{thm}
The reason that the term $\varrho^\alpha(\varphi)$ does not show up on the right-hand side
of \eqref{eqn:thm1_intro_Liouville_fillable} is due to the fillability of $\Sigma$ which simplifies the compactness proof. 
\begin{cor}
\label{cor:main_theorem_1_Liouville_fillable}
Suppose $(\Sigma ,\xi)$ is Liouville fillable and $\alpha\in \mathcal{C}(\Sigma,\xi)$. Let $\varphi$ be a non-trivial non-negative loop with contact Hamiltonian $h^{\alpha}$. Then
\begin{equation}
\label{eq:main_theorem_1_eq}
 \| h^{\alpha} \|_+ \ge \min\{ \wp( \alpha, \nu_0), \wp( \alpha , \nu_{\varphi}) \}
\end{equation}
and thus on Liouville fillable contact manifolds there are no $C^0$-small positive loops.
\end{cor}
\begin{proof}
Since $h^{\alpha} $ is non-negative one has $ \| h^{\alpha } \|_{\mathrm{osc}} \leq \| h^{\alpha} \|_+$. Thus if $\| h^{\alpha} \|_+ <  \wp( \alpha, \nu_0)$ then the second part of Theorem \ref{thm:thm1_intro_Liouville_fillable} implies that $\wp( \alpha , \nu_{\varphi}) \le \|h^{\alpha} \|_+$.
\end{proof}

Apart from providing lower bounds on the size of positive loops, Theorem \ref{thm:thm1_intro_Liouville_fillable} and Corollary \ref{cor:main_theorem_1_Liouville_fillable} are also quantitative existence results for closed Reeb orbits (the Weinstein Conjecture) in the presence of a non-trivial non-negative loop. This extends our earlier article \cite{AlbersFuchsMerry2015}. We illustrate this for the sphere $S^{2n-1}$. One direction is a famous result of Rabinowitz \cite{Rabinowitz1979}.

\begin{cor}
Suppose $\Sigma \subset \R^{2n}$ is a hypersurface which bounds a star-shaped region and is contained in the region between the two spheres $\partial B(r)$ and $\partial B(R)$ of radius $0 < r <R$. Then $\Sigma$ admits a closed characteristic of period $T>0$ such that  $ \pi r^2 \le T \le \pi R^2$.
\end{cor}
\begin{proof}
We recall from above that we have normalised the standard contact form $\alpha$ on $S^{2n-1}$ so that its Reeb flow $\theta^{\alpha}_t$ is one-periodic. Thus the assumptions of the Corollary translate to considering a contact form $f \alpha$ on $S^{2n-1}$ with $f :S^{2n-1} \to [\pi r^2, \pi R^2]$. The contact Hamiltonian of $\theta_t^{\alpha}$ with respect to the contact form $f \alpha $ is  the function $f \circ \theta^{\alpha}_{-t}$, see \eqref{eq:how_h_changes} below. 

There are two options. Either $\wp( f \alpha ,\nu_0) \le \pi(R^2 - r^2)$ or $\wp( f \alpha , \nu_0) > \pi(R^2 - r^2) $.  In the former case there is nothing to prove, since by iterating the orbit with period  $\wp( f \alpha ,\nu_0)$ we obtain one with period $\pi r^2 \le T \le \pi R^2$. Thus assume the latter. Then since $ \| f \|_{ \mathrm{osc}} 
\le \pi(R^2 - r^2)$, Theorem \ref{thm:thm1_intro_Liouville_fillable} is applicable, and we deduce the existence of a closed contractible Reeb orbit with period $T >0$ such that $\pi r^2 \le T \le \pi R^2$. 
\end{proof} 
\begin{rem}
With a little more work it should be possible to extend Theorem \ref{thm:thm1_intro_Liouville_fillable} to weaker notions of fillability such as $\Sigma$ being of contact type with a semi-positive filling. However, we won't pursue this in this article.

A similar argument as in Corollary \ref{cor:main_theorem_1_Liouville_fillable} works in the following setting. Suppose  $\Sigma$ is a  fibrewise starshaped hypersurface in a negative line bundle, which is either hypertight or suitably fillable. Assume that $\Sigma$ is ``squeezed'' between two circle bundles $S_r$ and $S_R$ of radius $0 < r <R$ (these correspond to prequantisation spaces as described above). Then either $\Sigma$ admits a closed contractible Reeb orbit with period less than $R^2-r^2$, or there exists a closed Reeb orbit in the free homotopy class $\nu_{\theta}$ of the Reeb flow $\theta_t$ (which is automatically a loop) with period in between $r^2$ and $R^2$. In this setting, the existence of a closed Reeb orbit is already known (it can be deduced from our earlier article \cite{AlbersFuchsMerry2015}, for example), but we believe the period bounds are new. This should be compared to work of Gutt \cite[Theorem 1.6]{Gutt2015}, who under additional hypotheses also deduced multiplicity results. 
\end{rem}

\subsection*{Minimal positive loops}

The following result shows that certain periodic Reeb flows are the \emph{unique} minimizers of the $\|\cdot\|_+$-norm 
among positive loops with a fixed associated free homotopy class. It applies in particular 
to periodic Reeb flows whose associated free homotopy class is not torsion. We recall that for a contact form $\alpha$ we denote by $\theta^\alpha$ the positive path $t\rightarrow \theta_{t}^\alpha$ given by the Reeb flow.
\begin{thm}\label{thm:minimal_periodic}
Suppose $(\Sigma,\xi)$ is hypertight (resp. Liouville fillable) and $\alpha\in \mathcal{C}(\Sigma,\xi)$ admits no contractible Reeb orbits (resp. is induced from a Liouville filling). 
Assume that the Reeb flow $\theta^\alpha$ is 1-periodic and $\wp(\alpha,\nu_0)\geq1$. Then any positive loop $\varphi$ with $\nu_\varphi=\nu_{\theta^\alpha}$
satisfies 
\begin{equation}
\|h^\alpha\|_+\ge 1
\end{equation} 
and equality holds if and only if $\varphi=\theta^\alpha$. 
\end{thm}

\begin{rem} 
The assumptions $\wp(\alpha,\nu_0)\geq1$ and $\theta^\alpha$ being $1$-periodic imply that $\wp(\alpha,\nu_{\theta^\alpha})\geq1$. Indeed, if we assume that $0<\delta:=\wp(\alpha,\nu_{\theta^\alpha})<1$ we find a point $x\in \Sigma$ lying on a $\delta$-periodic Reeb orbit in class $\nu_{\theta^\alpha}$. Moreover, $x$ lies  on a $1$-periodic Reeb orbit in class $\nu_{\theta^\alpha}$ since the Reeb flow is $1$-periodic. This yields a $(1-\delta)$-periodic Reeb orbit through $x$ in class $\nu_0$ which contradicts the assumption $\wp(\alpha,\nu_0)\geq1$. 

We conclude $\wp(\alpha,\nu_0)\geq1$ and $\wp(\alpha,\nu_{\theta^\alpha})\geq1$ and thus the inequality is already contained in Corollary \ref{thm:hypertight_small} resp.~Corollary \ref{cor:main_theorem_1_Liouville_fillable}. The equality case requires an additional argument, though.

As an example consider the standard sphere $(S^{2n-1},\xi)$ with contact form $\alpha$ whose Reeb flow is periodic with $\wp(\alpha,\nu_0)=1$. Then the above result shows in particular that the contact Hamiltonian $h^\alpha$ of any positive loop $\varphi$ on $(S^{2n-1},\xi)$
satisfies $\|h^\alpha\|_+\geq 1$ with equality if and only if $\varphi$ coincides with the Reeb flow.

The assumption that the Reeb flow is periodic with $\wp(\alpha,\nu_0)\geq1$ is important. The rational ellipsoid $E(\tfrac12,1)=\{(z,w)\in \C^2\mid \pi(2|z|^2+|w|^2)=1\}$ gives rise to a contact form $f\alpha$ on $S^3$ with $f$ taking values in $[1,2]$. Its Reeb flow is 1-periodic but the exceptional orbit $\{w=0\}$ gives rise to $\wp(f\alpha,\nu_0)=\tfrac12$. Now we consider the Reeb flow of $\alpha$ as a positive loop $\varphi$ on the contact manifold $(S^3,f\alpha)$. Its contact Hamiltonian is $h^{f\alpha}=\tfrac1f$. Thus $\|h^{f\alpha}\|_+=1$ and $\varphi\neq\theta^{f\alpha}$. Therefore, Theorem \ref{thm:minimal_periodic} becomes false if we drop the assumption $\wp(f\alpha,\nu_0)\geq1$.
\end{rem}

\begin{proof}
As pointed out it is sufficient to focus on the case of equality, since the inequality is already contained in Corollary \ref{thm:hypertight_small} resp.~Corollary \ref{cor:main_theorem_1_Liouville_fillable}. Thus we assume $\varphi$ is a positive loop with $\nu_\varphi=\nu_{\theta_1^\alpha}$ and contact Hamiltonian $h^\alpha$ satisifying $\|h^\alpha\|_+=1$. Moreover, from now on we will drop the superscript $\alpha$ and recall that
\begin{equation}
  \| h \|_+ = \int_0^1 \max_{x \in \Sigma } h_t(x) \,dt.
\end{equation}
We define the diffeomorphism $\tau:[0,1]\to[0,1]$ by
\begin{equation}
\tau(t):=\frac{\int_0^t\max_{x \in \Sigma } h_s(x)\,ds}{\int_0^1\max_{x \in \Sigma } h_s(x)\,ds}=
\frac{1}{\|h\|_+}\int_0^t\max_{x \in \Sigma } h_s(x)\,ds
\end{equation}
and consider the loop $\{\hat\varphi_t\}_{t\in S^1}$ defined by
\begin{equation}
\hat\varphi_{t}:=\varphi_{\tau^{-1}(t)}.
\end{equation}
The contact Hamiltonian $h_t$ transforms appropriately under this reparametrisation 
of $\varphi=\{\varphi_t\}$, i.e.~$\hat\varphi:=\{\hat\varphi_t\}$ has contact Hamiltonian 
$\hat h_t(x)=(\tau^{-1})'(t)\cdot h_{\tau^{-1}(t)}(x)$ and therefore $\hat\varphi$ still is a positive loop. Moreover, 
the reparametrisation guarantees that we have
for all $t\in S^1$ the equality
\begin{equation}\label{eqn:max_h=1}
\max_{x\in\Sigma} \hat h_t(x)=\frac{\|h\|_+}{\displaystyle\max_{x\in\Sigma}h_{\tau^{-1}(t)}(x)}\cdot \max_{x\in\Sigma}h_{\tau^{-1}(t)}(x)=\|h\|_+.
\end{equation}
Of course, we have $\|\hat h\|_\pm=\| h\|_\pm$ and therefore
\begin{equation}\label{eqn:h_hat_and_h_have_same_osc_norm}
\|\hat h\|_\mathrm{osc}=\| h\|_\mathrm{osc}\;.
\end{equation}
The contact Hamiltonian
\begin{equation}
k_t(x):=1-\hat h_{-t}(x)
\end{equation}
generates
\begin{equation}
\psi_t=\theta_{t}\circ\hat\varphi_{-t}\;.
\end{equation}
Since the Reeb flow and $\hat\varphi$ are loops so is $\psi=\{\psi_t\}$. Moreover, inequality \eqref{eqn:max_h=1} together with the positivity of $\varphi$ implies
\begin{equation}\label{eqn:max_k<1}
0\leq k_t(x)<1
\end{equation} 
and 
\begin{equation}
\|k\|_\mathrm{osc}=\|\hat h\|_\mathrm{osc}=\|\hat h \|_+-\|\hat h \|_-<1\;.
\end{equation}
It follows from the assumption that $\nu_\psi=\nu_0$. Thus, using Corollary \ref{thm:hypertight_small} resp.~Theorem \ref{thm:thm1_intro_Liouville_fillable} it follows that the non-negative loop $\psi$ 
is necessarily trivial; thus we obtain the equality $\varphi=\theta_{1}$ as claimed. 
\end{proof}

\subsection*{Contact systolic inequalities}
In the following we assume that $(\Sigma, \xi)$ admits a positive loop $\varphi$ of contactomorphisms. Corollary \ref{cor:main_theorem_1} implies the following $L^{\infty}$-contact systolic inequality, where the $L^{\infty}$ refers to the function $f$. 
\begin{cor}\label{cor:ugly_systolic_inequality}
For all smooth functions $f : \Sigma \to (0, \infty)$, one has
\begin{equation}
   \min\bigg\{ \frac{\min f}{\max f} \cdot \varrho^{\alpha}(\varphi) \cdot\wp( f\alpha, \nu_0), \wp(f \alpha , \nu_{\varphi}) \bigg\} \le \| h^{\alpha} \|_+ \cdot \max f
\end{equation}
where $h^\alpha$ is the contact Hamiltonian of the positive loop $\varphi$ with respect to a contact form $\alpha$, and $\varrho^{\alpha}(\varphi)$ was defined in \eqref{eq:pi_alpha_phi}. 
\end{cor}
We point out that unless $f$ is a constant function, there is typically no relation between $\wp(\alpha , \nu)$ and $\wp( f\alpha, \nu) $. 
\begin{proof}
Since 
\begin{equation}
\label{eq:how_h_changes}
  h^{f \alpha}_t(x) = f \big(\varphi_t^{-1}(x) \big) h^{\alpha}_t(x)
\end{equation}
and 
\begin{equation}
\label{eq:how_rho_changes}
\rho^{f \alpha}_t(x) = \frac{f (\varphi_t(x))}{f(x)}\rho^{\alpha}_t(x)
\end{equation}
the assertion immediate follows from Corollary \ref{cor:main_theorem_1} applied to the contact form $f\alpha$.
\end{proof} 
The systolic inequality simplifies in certain situations.
\begin{cor}
Suppose $(\Sigma, \xi)$ is hypertight, and $\alpha \in \mathcal{C}(\Sigma, \xi)$ has no contractible Reeb orbits. Let $\varphi$ denote a positive loop with contact Hamiltonian  $h^{\alpha}$. Then  for any smooth function $f : \Sigma \to (0 , \infty)$, one has
\begin{equation}
\wp(f \alpha , \nu_{\varphi} ) \le  \| h^{\alpha} \|_+  \cdot \max f.
\end{equation}
\end{cor}

\begin{cor}
Suppose $(\Sigma, \xi)$ is Liouville fillable with filling $(W,d \lambda)$. Let $\alpha := \lambda|_{\Sigma}$. Suppose $\varphi$ denotes a positive loop, and denote its contact Hamiltonian with respect to $\alpha$ by $h^{\alpha}$. 
Then for all smooth functions $f : \Sigma \to (0, \infty)$, one has
\begin{equation}
   \min\big\{ \wp( f\alpha, \nu_0), \wp(f \alpha , \nu_{\varphi}) \big\} \le \| h^{\alpha} \|_+ \cdot\max f\end{equation}
where $h^\alpha$ is the contact Hamiltonian of the positive loop $\varphi$ with respect to a contact form $\alpha$.
\end{cor}

\begin{cor}
Suppose $(\Sigma, \xi)$ admits a contact form $\alpha$ with one-periodic Reeb flow $\theta$. Then for any smooth function $f : \Sigma \to (0 , \infty)$, one has
\begin{equation}
   \min\bigg\{ \frac{\min f}{\max f} \cdot \wp( f\alpha, \nu_0), \wp(f \alpha , \nu_{\theta}) \bigg\} \le  \max f.
\end{equation}
Thus if $\nu_{\theta} = \nu_0$ (for instance, if $\Sigma$ is simply connected), one has 
\begin{equation}
  \wp ( f \alpha ; \nu_0) \le \frac{ (\max f)^2}{\min f}.
\end{equation}
\end{cor}

Inspired by systolic inequalities in Riemannian geometry, there have been attempts to prove a relationship between the minimal Reeb period $\wp(\alpha)$ and the contact volume $\text{vol}(\alpha) := \int_{\Sigma} \alpha \wedge d \alpha^{n-1}$. This corresponds to using the $L^n$-norm of the function $f$ instead of the $L^{\infty}$-norm: 
\begin{equation}
  \text{vol}(f \alpha) = \int_{\Sigma} (f \alpha) \wedge d( f \alpha)^{n-1 } = \int_{\Sigma} f^n \alpha \wedge d \alpha^{n-1}.
\end{equation}
Nevertheless, recent work by Abbondandolo-Bramham-Hryniewicz-Salom\~ao \cite[Theorem 2]{AbbondandoloBramhamHryniewiczSalomao2015} shows that for $S^3$  equipped with the standard contact structure no such $L^2$-contact systolic inequality holds. However in \cite[Theorem 1]{AbbondandoloBramhamHryniewiczSalomao2015} it is proved that a $L^2$-contact systolic inequality holds in a $C^3$-neighbourhood of the Zoll contact forms on $S^3$, which in particular includes the standard contact form. An earlier result in this direction is contained in \cite{AlvarezPaivaBalacheff2014}. 

We prove all of the above contact systolic inequalities under the assumption of the existence of a positive loop. An unconditional systolic inequality would immediately imply the Weinstein Conjecture. Thus it seems to us that such a conditional statement is the best we can reach with present technology.

\subsection*{Acknowledgements}
PA is supported by the SFB 878 - Groups, Geometry and Actions and SFB/TRR 191 - Symplectic Structures in Geometry, Algebra and Dynamics. UF is supported by the SNF fellowship 155099 and a fellowship at Institut Mittag-Leffler. We thank Alberto Abbondandolo, Leonid Polterovich, Sheila Sandon, and Egor Shelukhin for helpful comments.

\section{Preliminaries} 
\label{sec:preliminaries}

Let $(\Sigma,\xi)$ be a closed connected cooriented contact manifold. We fix once and for all a supporting contact form $\alpha$, i.e.~$\alpha \in \mathcal{C}(\Sigma, \xi)$. In contrast to the Introduction, from now on we will drop the superscript $\alpha$ in our notation. We will prove Theorem \ref{thm:thm1_intro} for this fixed choice of $\alpha$ as Theorem \ref{thm:theorem_1_in_the_paper}.

Let $\varphi = \{ \varphi_t \}_{t \in [0,1]}$ denote a path of contactomorphisms with $\varphi_0 = \mathrm{id}$.   The \emph{contact Hamiltonian} of $\varphi$ with respect to $\alpha$ is the function $h : \Sigma \times [0,1] \to \R$ given by 
\begin{equation}
  h_t(\varphi_t(x) ) = \alpha \left( \frac{d}{dt}\varphi_t(x) \right) 
\end{equation}
The conformal factor of $\varphi$ with respect to $\alpha$ is the positive function 
\begin{equation}
  \rho : \Sigma \times [0,1] \to (0 , \infty)
\end{equation}
implicitly defined by the requirement that $\varphi_t^* \alpha = \rho_t \alpha$. 
We define 
\begin{equation}
\label{eq:inf_rho}
\varrho = \varrho(\varphi)= \varrho(h) := \min_{(x ,t) \in \Sigma \times [0,1]} \rho_t(x), \qquad \Pi = \Pi(\varphi) := \max_{(x ,t) \in \Sigma \times [0,1]} \rho_t(x)
\end{equation}
Note that since $\varphi_0 = \mathrm{id}$, one has $\rho_0 \equiv 1$ and hence $\varrho \le 1 \le \Pi$.

In all of the following we will implicitly assume that the contact Hamiltonian $h_t$ is 1-periodic in $t$. This can be achieved by reparametrising the path $\varphi_t$ in time.  We will need the following easy lemma.
\begin{lem}
\label{lem:inverse}
If we denote by $\bar{h}_t$ the contact Hamiltonian and by $ \bar{\rho}_t$ the conformal factor of the path $t \mapsto \varphi_t^{-1}$, one has
\begin{equation} \begin{aligned}
\bar{\rho}_t(x) &= \frac{1}{\rho_t(\varphi^{-1}_t(x))}, \\  
\bar{h}_t(x) & = - \frac{h_t(\varphi_t(x))}{\rho_t(x)}.
  \end{aligned}\end{equation} 
\end{lem}
\begin{proof}
We apply $(\varphi_t^{-1})^*$ to the equation $\varphi_t^*\alpha = \rho_t \alpha$: 
\begin{equation}
 \alpha = (\varphi_t^{-1})^* \varphi_t^* \alpha = (\varphi_t^{-1})^*(\rho_t \alpha) = (\rho_t \circ \varphi_t^{-1}) (\varphi_t^{-1})^* \alpha ,
 \end{equation}
 and thus 
 \begin{equation}
 (\varphi_t^{-1})^* \alpha = \frac{1}{\rho_t \circ \varphi_t^{-1}} = \bar{\rho}_t \alpha.
 \end{equation}
 Set $X_t(\varphi_t(x)) : = \frac{d}{dt}\varphi_t(x)$ and similarly $\bar{X}_t(\varphi_t^{-1}(x)) = \frac{d}{dt}\varphi_t^{-1}(x)$. To compute $\bar{h}_t$ we begin with 
\begin{equation}
0 = \frac{d}{dt}( \varphi_t^{-1}\circ \varphi_t) (x) 
  = \bar{X}_t(x) + D \varphi_t^{-1} (X_t(\varphi_t(x)))
  \end{equation}
Now  applying $\alpha$ to both sides we obtain
\begin{equation} \begin{aligned}
\alpha_x (\bar{X}_t(x)) & = - \alpha_{\varphi_t(x)} \bigg( D \varphi_t^{-1} \Big(X_t\big(\varphi_t(x)\big)\Big)\bigg)  \\
& = - \bar{\rho}_t\big( \varphi_t(x)\big) \alpha_{\varphi_t(x)} \big(X_t(\varphi_t(x)) \big)\\
& =  - \frac{1}{\rho_t(x)} h_t( \varphi_t(x)) \\
& \stackrel{\mathrm{def}}{=}\bar{h}_t(x)
\end{aligned}\end{equation} 
\end{proof} 

\begin{defn}\label{defn:symplectisation}
The symplectisation of $(\Sigma,\alpha)$ is the manifold $S \Sigma := \Sigma \times (0, \infty)$, equipped with the symplectic form $\omega = d(r \alpha)$, where $r$ is the coordinate on $(0,\infty)$. Where convenient, we will use the letter $z$ to denote a point $(x,r) \in S \Sigma$ and $\pi_{(0,\infty)}$ to denote the projection $S\Sigma=\Sigma\times (0,\infty)\rightarrow (0,\infty)$ onto the second factor. \\
In some of the statements that follow we assume $\Sigma$ is Liouville fillable. In this case we always denote by $(W , \omega = d \lambda)$ the filling, so that $\alpha = \lambda|_{\Sigma}$. We denote by $\widehat W:=W\cup\big(\Sigma\times[1,\infty)\big)$ the completion of $W$. Note that $ S \Sigma $ embeds inside $\widehat W$. \\
For an arbitrary symplectic manifold $(V,\Omega)$ and a smooth function $H:V\rightarrow \R$, we denote by $X_H$ the Hamiltonian 
vector field on $V$ associated to $H$, defined by the relation $-dH=\Omega(X_H,\cdot)$. If $H$ depends on additional parameters, we obtain by the same procedure a parameter-dependent Hamiltonian vector field on $V$ also denoted by $X_H$.
\end{defn}

\begin{defn}\label{defn:ham_diffeo}
We denote by $H_t : S \Sigma \to \R$ the function 
\begin{equation}
H_t(x,r) = r h_t(x).
\end{equation}
This is the Hamiltonian of the symplectomorphism $\phi_t : S \Sigma \to S \Sigma$ given by 
\begin{equation}
\phi_t(x,r) = \left( \varphi_t(x), \frac{r}{\rho_t(x)} \right) .
\end{equation} 
Similarly we denote by 
\begin{equation}\label{eq:phip}
\bar{H}_t(x,r) := r \bar{h}_t(x) = -\frac{r}{\rho_t(x)}h_t(\varphi_t(x)),
\end{equation}
which is the Hamiltonian of the symplectomorphism $\phi_t^{-1}$, cf. Lemma \ref{lem:inverse}.
\end{defn}

Our aim now is to modify the Hamiltonian function $\bar{H}_t$ to a related function $\mathcal{H}_t$ by using a collection of cutoff functions. We introduce these functions now. In fact, the functions $m$ and $\chi$ defined below won't actually be used until Definition \ref{defn:rab_functions}, but for the purposes of clarity we collect all the cutoff functions into one place here. 
For some discussion on their significance we refer to remark \ref{rem:cutoffs}.
\begin{defn}[The various cutoff functions]
\label{defn:cutoff_functions}
Fix a (small) constant $\delta >0$, which later on will be specified in the proof of Theorem \ref{thm:technical_theorem_about_paths}. 
\begin{enumerate}
  \item \textbf{The function $m_\delta$:}
  Fix a smooth function  $m_\delta : (0,\infty) \to \R$ such that
\begin{equation}
m_\delta(r)=
\begin{cases}
r- 1 , & r\in [e^{-\delta} , e^{\delta}],\\
e^{2 \delta} - 1, & r\in [ e^{2 \delta}   ,\infty),\\
e^{ - 2 \delta} - 1, & r\in(0, e^{- 2 \delta} ],
\end{cases}\qquad \text{and} \qquad  m'_\delta(r)\geq0.
\end{equation}
Note that $m'_\delta(1)=1$ and $m_\delta^{-1}(0)=\{1\}$.
\item \textbf{The function $\beta_\delta$:}
Fix a smooth function  $\beta_\delta : (0,\infty) \to [0,1]$ such that
\begin{equation}
  \beta_\delta(r) = 
  \begin{cases}
  1, & r \in [e^{ - \delta}, e^{\delta}], \\
  0, & r \notin [e^{ -  2 \delta}, e^{ 2 \delta}].
  \end{cases}
  \end{equation}
\item \textbf{The function $\chi$:}
Fix a smooth function $\chi : S^1 \to \R$ such that
\begin{equation}
  \chi(t) = 0, \qquad \forall \, t \in [\tfrac12,1] \subset S^1, \qquad \text{and} \qquad \int_{S^1}\chi (t) \,dt = 1.
\end{equation}
\item \textbf{The function $\kappa$:}
Fix a smooth function $\kappa : [0,1] \to [0,1]$ such that
\begin{equation}
\dot  \kappa(t) \ge 0, \qquad \text{and} \qquad  \kappa(\tfrac12) = 0,  \qquad \text{and} \qquad \kappa(1) = 1.
\end{equation}

\item \textbf{The functions $f_{\infty}^{\pm}$ and $f_{\sigma}$:} 

We fix two smooth functions $f^\pm_\infty\in C^\infty(\R,[0,1])$ satisfying
\begin{enumerate}
\item $f^+_\infty(s)=0$ for $s\leq-1$, $f^+_\infty(s)=1$ for $s\geq0$ and $f^+_\infty$ is monotone increasing,
\item $f^-_\infty(s)=1$ for $s\leq0$, $f^-_\infty(s)=0$ for $s\geq1$ and $f^-_\infty$ is monotone decreasing.
\end{enumerate}
For $\sigma\geq 0$, we fix a smooth family of functions $f_\sigma\in C^\infty(\R,[0,1])$ satisfying
\begin{enumerate}
\setcounter{enumi}{2}
\item $f_\sigma(s)=0$ for $s\leq-\sigma-1$ and $s\geq \sigma+1$ 
\item $0\leq f'_\sigma(s)\leq2$ on $(-\sigma-1,-\sigma)$ and $0\geq f'_\sigma(s)\geq-2$ on $(\sigma,\,\sigma+1)$,
\item for $\sigma\geq1$: $f_\sigma(s)=1$ for $s\in[-\sigma,\sigma]$,
\item for $\sigma\leq1$: $f_\sigma'(s)=0$ for $s\in[-\sigma,\sigma]$, $\lim_{\sigma\to0}f_\sigma=0$ in the strong $C^\infty$-topology,
\item\label{limit_of_betas} $\lim_{\sigma\to\infty}f_\sigma(s-\sigma)=f_\infty^+(s)$ and $\lim_{\sigma\to\infty}f_\sigma(s+\sigma)=f_\infty^-(s)$ in the $C^\infty_{\mathrm{loc}}$-topology. 
\end{enumerate} 
These functions are illustrated in Figure \ref{pic:beta}.
\end{enumerate}

\begin{rem}
The functions $\beta_\delta$, $f_\infty^\pm$, and $f_\sigma$ are the real heroes.
\end{rem}

\begin{figure}[ht] 
\def\svgwidth{80ex} 
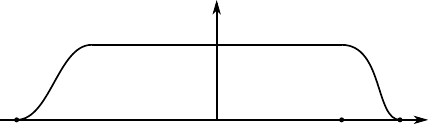\\[3ex]
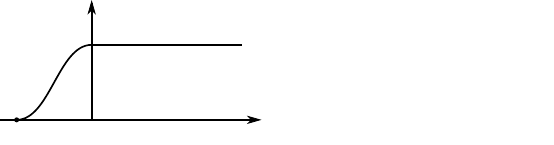
\caption{The functions $f_\sigma$ and $f^\pm_\infty$.}\label{pic:beta}
\end{figure}
\end{defn}

Recall the definition of the projection $\pi_{(0,\infty)}:S\Sigma\rightarrow (0,\infty)$ from Definition \ref{defn:symplectisation}.

\begin{defn}
\label{defn:mathcalH}
We define a new function $\mathcal{H}_t : S \Sigma \to \R$:
\begin{align}
\mathcal{H}_t(x,r) & = \dot \kappa(t) \beta_\delta (\pi_{(0,\infty)}(\phi_{\kappa(t)}(x,r))) \bar{H}_{\kappa(t)}(x,r) \label{eq:def_of_K} \\
& = - \dot \kappa(t) \beta_\delta \left( \frac{r}{\rho_{\kappa(t)}(x)} \right) \frac{r}{\rho_{\kappa(t)}(x) } h_{\kappa(t)}(\varphi_{\kappa(t)}(x)). \nonumber
\end{align}
Finally we incorporate the family of cutoff functions $f_{\sigma}$ and $f_{\infty}^{\pm}$ and set
\begin{equation}
\label{eq:K_with_f}
  \mathcal{H}^{\sigma}_{s,t}(x,r) : =  f_{\sigma}(s)\mathcal{H}_t(x,r)  , 
\end{equation}
and
\begin{equation}
\label{K_with_f_pm}
\mathcal{H}^{\pm}_{s,t}(x,r) : =  f_{\infty}^{\pm}(s)\mathcal{H}_t(x,r)
\end{equation}
\end{defn}

\begin{rem}
\label{rem:motivating_K}
We point out that all Hamiltonian functions are compactly supported inside the symplectisation, for instance note that for each $t \in [0,1]$,
\begin{equation}
\label{eq:supp}
  \mathrm{supp}(\mathcal{H}_t)  \subset \left\{ (x,r) \in S \Sigma \, \bigg|\, \frac{r}{\rho_{\kappa(t)}(x)}  \in [e^{- 2 \delta}, e^{2 \delta}] \right\} \subset  \Sigma \times [ \varrho e^{-2 \delta}, \Pi e^{2 \delta}],
\end{equation}
where $\varrho, \Pi$ were defined in \eqref{eq:inf_rho}.

The definition of $\mathcal{H}_t$ may appear somewhat obscure, so let us unwrap it slightly.
Recall first from Definition \ref{defn:ham_diffeo} that the Hamiltonian $\bar{H}_t$ generates $\phi_t^{-1}$, while $\phi_t$ is generated by $H_t$. 
Considering the function $\hat{\mathcal{H}}_t(z) = \beta_\delta(\pi_{(0,\infty)}(\phi_t(z)))\bar{H}_t(z)$ for $z=(x,r)\in S\Sigma$ is a trick
 we learnt from the paper \cite[Lemma 41]{Shelukhin2014}. The key observation is that if a point $z\in S\Sigma$ lies in the interior of the region
 where $\beta_\delta$ is equal to $1$ then the associated Hamiltonian flows (cf. Definition \ref{defn:ham_diffeo}) of $\hat{\mathcal{H}}$ and $\bar{H}$ agree along the orbit starting at
 $z$. 
 
Indeed, $\hat{\mathcal{H}}_t(\phi_t^{-1}(z))=\bar{H}_t(\phi_t^{-1}(z))$ for any $z=(x,r)\in \Sigma\times (e^{-\delta},e^\delta)$ since $\beta_\delta(\phi_t(\phi_t^{-1}(z)))=1$; thus the associated Hamiltonian vector fields (cf. Definition \ref{defn:symplectisation}) coincide on a neighborhood of $\big\{(\phi_t^{-1}(y,1),t)\in S\Sigma\times [0,1]\mid y\in \Sigma\big\}$.

Therefore the flow $\Phi_t^{\hat{\mathcal{H}}}$ of $\hat{\mathcal{H}}$ coincides, when applied to a neighborhood of $\Sigma\times \{1\}$, with the flow $\phi^{-1}_t$ of $\bar{H}_t$ by Lemma \ref{lem:inverse}. Since
the function $\mathcal{H}$ defined above differs from $\hat{\mathcal{H}}$ only a time reparametrisation function $\kappa$, we find that the Hamiltonian flow $\Phi_t^{\mathcal{H}}$ of $\mathcal{H}$ satisfies 
$\Phi_1^{\mathcal{H}}(y,1)=\phi_1^{-1}(y,1)$ for all $(y,1)\in S\Sigma$.
As in \cite{Shelukhin2014}, this leads to an improved action estimate over our previous paper \cite{AlbersFuchsMerry2015}.

If $\Sigma$ is Liouville fillable then we extend all the above Hamiltonian functions by constants over $\widehat W \setminus S \Sigma$. This is possible since the Hamiltonian functions are all compactly supported inside the symplectisation.
\end{rem}

\begin{defn}
\label{defn:pm_norma}
Given a compactly supported function $F: Q \times [0,1] \to \R$ on a manifold $Q$, we denote by 
\begin{equation}
  \| F \|_+ := \int_0^1 \max_{q\in Q } F_t(q) \,dt. \qquad \| F\|_- := \int_0^1 \min_{q \in Q} F_t(q)\,dt,
\end{equation}
Finally we set
\begin{equation}
  \| F \|_{\mathrm{osc}} := \| F\|_+ - \|F\|_-.
\end{equation}
We note that the quantity $ \| F\|_{\mathrm{osc}}$ is always non-negative
while $ \|F\|_{\pm}$ is non-negative for non-negative functions $F$.
\end{defn}

Recall from Definition \ref{defn:mathcalH} the construction of the Hamiltonian function $\mathcal{H}_t$ on $S\Sigma$ (resp. on $\widehat{W}$) from a contact Hamiltonian $h_t$ on $\Sigma$ and a cutoff function $\beta_\delta$ depending implicitly on $\delta$. In the following we always assume that $h_t(x)\geq0$ and thus $\mathcal{H}_t(x,r)\leq 0$ for all $(x,r)\in S\Sigma$ and $t\in[0,1]$. This is mostly for convenience since many results generalize appropriately. The following lemma is crucial in all what follows. 
\begin{lem}
\label{lem:Kosc_and_h}
Suppose $h_t(x) \ge 0$ for all $x \in \Sigma$ and $t \in [0,1]$. Then one has the estimate
\begin{equation}
   \| \mathcal{H} \|_{\mathrm{osc}} \le e^{2 \delta} \| h \|_+.
\end{equation}
\end{lem}

\begin{proof}
We recall the support of $\mathcal{H}$ from \eqref{eq:supp}. Moreover we have
$ \| \mathcal{H} \|_+ = 0$ since $\mathcal{H}_t \le 0$ and $\mathcal{H}_t$ is compactly supported in $S\Sigma$. We now estimate 
$\| \mathcal{H} \|_-$:
\begin{equation} \begin{aligned}
-\| \mathcal{H} \|_- & = - \int_0^1 \min_{(x,r) \in S \Sigma} \mathcal{H}_t(x,r)\,dt \\
& = - \int_0^1 \min_{(x,r) \in S \Sigma}  \left( -\dot \kappa(t) \beta_\delta \left( \frac{r}{\rho_{\kappa(t)}(x)} \right) \frac{r}{\rho_{\kappa(t)}(x) } h_{\kappa(t)}(\varphi_{\kappa(t)}(x)) \right) \,dt \\
& = \int_0^1 \max_{(x,r) \in S \Sigma}  \left( \dot \kappa(t) \beta_\delta \left( \frac{r}{\rho_{\kappa(t)}(x)} \right) \frac{r}{\rho_{\kappa(t)}(x) } h_{\kappa(t)}(\varphi_{\kappa(t)}(x)) \right) \,dt \\
& \stackrel{(\mathrm{a})}{=} \int_0^1 \max_{(x,r) \in S \Sigma}  \left(  \beta_\delta \left( \frac{r}{\rho_{t}(x)} \right) \frac{r}{\rho_{t}(x) } h_{t}(\varphi_{t}(x)) \right) \,dt \\
& \stackrel{(\mathrm{b})}{\le} \int_0^1 \max \left\{   \frac{r}{\rho_{t}(x) } h_{t}(\varphi_{t}(x)) \;  \bigg| \; \frac{r}{\rho_t(x)} \in [e^{-2 \delta }, e^{2 \delta}] \right\}     \,dt \\
& \le e^{2 \delta}\int_0^1  \max_{x \in \Sigma} h_t(\varphi_t(x))\,dt \\
& = e^{2 \delta}\| h \|_+,
\end{aligned}\end{equation} 
where in (a) we  removed the time reparametrisation function $\kappa$, and in (b) 
we used the definition of $\beta_\delta$ and the fact that $\rho_t$ and $h_t$
 are non-negative.
\end{proof} 

We abbreviate $\Lambda(S \Sigma) := C^{\infty}_{\mathrm{contr}}(S^1,  S \Sigma)$ the space of smooth contractible loops $z = (x,r) :S^1 \to S \Sigma$ endowed with the Whitney $C^\infty$-topology.
 \begin{defn}
 \label{defn:rab_functions}
  We define four different Rabinowitz action functionals: 
\begin{equation} \begin{aligned}
\mathcal{A}_0 & :  \Lambda(S \Sigma) \times \R \to \R \\
\mathcal{A}_{ \mathcal{H}} & :  \Lambda(S \Sigma) \times \R \to \R \\
  \mathcal{A}^{\sigma}_s&  : \Lambda(S \Sigma) \times \R \to \R \\  \mathcal{A}^{\pm}_s & :  \Lambda(S \Sigma) \times \R \to \R.
\end{aligned}
\end{equation} 
The simplest one $\mathcal{A}_0$ is given by
\begin{equation}
\label{eq:A_varpi}
  \mathcal{A}_0 ( x, r, \eta) := \int_{S^1}r \alpha( \dot  x) - \eta \int_{S^1} \chi(t) m_\delta\big(r(t)\big)\,dt ,
\end{equation}
The other three are obtained from $\mathcal{A}_0$ as follows:
\begin{equation}
\label{eq:A_K}
  \mathcal{A}_{\mathcal{H}}( x, r, \eta) :=  \mathcal{A}_0(x,r, \eta) - \int_{S^1} \mathcal{H}_t \big(x(t),r(t) \big) \,dt,
\end{equation}
and 
\begin{equation}
\label{eq:A_sigma_s}
 \mathcal{A}^{\sigma}_s ( x, r, \eta) := \mathcal{A}_0(x, r ,\eta) - \int_{S^1} \mathcal{H}^{\sigma}_{s,t}\big(x(t),r(t)\big)\,dt,
\end{equation}
and finally 
\begin{equation}
\label{eq:A_pm}
  \mathcal{A}^{ \pm}_s( x, r, \eta) := \mathcal{A}_0(x,r, \eta)- \int_{S^1} \mathcal{H}^{\pm}_{s,t} \big(x(t),r(t) \big) \,dt.
\end{equation}
\end{defn}
Note  from Figure \ref{pic:beta} we have 
\begin{equation}
\label{eq:asymptotic_A_sigma}
  \mathcal{A}^{\sigma}_{s} = \mathcal{A}_0, \qquad \forall \, |s| \ge \sigma +1 . 
\end{equation}
and similarly 
\begin{align}
  \mathcal{A}^+_s = \mathcal{A}_0,  \qquad \forall \, s \le -1,& \qquad
    \mathcal{A}^+_s = \mathcal{A}_{\mathcal{H}}, \qquad \forall \, s \ge 0, \label{eq:asymptotic_A_+}\\
      \mathcal{A}^-_s = \mathcal{A}_{\mathcal{H}}, \qquad \forall \, s \le 0,& \qquad
        \mathcal{A}^-_s = \mathcal{A}_0,  \qquad \forall \, s \ge 1, \label{eq:asymptotic_A_-}
\end{align}

\begin{rem}\label{rem:loopspace_notation}
If $\Sigma$ has Liouville filling $W$ then we do not work on the symplectisation $S\Sigma$ but instead on the completion $\widehat W$. As mentioned above all Hamiltonian functions smoothly extend over $W$ by constants. Then the above Rabinowitz action functionals are defined on $ \Lambda(\widehat W) \times \R$ instead of $ \Lambda(S \Sigma) \times \R$. In the remainder of this section when the distinction is immaterial we will use the notation $\Lambda$ to indicate either $\Lambda(S \Sigma)$ or $\Lambda ( \widehat W)$.
\end{rem}

\begin{rem}\label{rem:cutoffs}
The second integrand in the functional $\mathcal{A}_0$ is modification of the $m(r)=r$ which is the Hamiltonian function which gives rise to level-wise Reeb flow on $S\Sigma$. The function $m_\delta$ is a cut off version of $m$ which still gives rise to level-wise Reeb flow near $\Sigma\times \{1\}$. The function $\chi$ is only really relevant in (\ref{eq:A_K}), where it ensures that the vector field $X_{\chi m_\delta}$ vanishes for $t\in [0,\frac{1}{2}]$, while $\kappa$ ensures the vanishing of $X_{\mathcal{H}^\sigma_{s,t}}$ for $t\in [\frac{1}{2},1]$. Together with the properties $m'_\delta(1)=1$ and $m_\delta^{-1}(0)=1$, this allows a concrete description of the critical points of the functionals in Lemma \ref{lem:critial_points_of_functionals}. One can think of $\eta$ as a Lagrange multipler.

Recall that $\mathcal{H}_{s,t}^\sigma:= f_\sigma \mathcal{H}_t$ where $f_0=0$ and $f_\sigma$ are compactly supported functions converging in $C^\infty_{loc}$ to the constant function $1$ as $\sigma\rightarrow \infty$. These functions yield interpolations $\mathcal{A}_s^\sigma$ between the unperturbed functional $\mathcal{A}_0$ and the perturbed functional $\mathcal{A}_\mathcal{H}$. A continuation map argument allows then to show the existence of suitable critical points of $\mathcal{A}_\mathcal{H}$ from critical points of $\mathcal{A}_0$. The function $\beta_\delta$ ensures (together with $m_\delta$) that for each flow line $(u,\eta)$ occuring in the continuation argument, the map $u$ is holomorphic wherever it maps into $\Sigma\times (0,e^{-\varepsilon})$. This in turn allows to use methods from SFT compactness to show that  the Hofer energy of these curves is bounded above by $e^{\varepsilon} \E(u,\eta)<\wp(\alpha,\nu_0)$ and thus they cannot break in the concave end of $S\Sigma$.
\end{rem}

More information on these functionals and their uses is contained in the survey article \cite{AlbersFrauenfelder2012a}. The following well-known lemma computes the critical points of $\mathcal{A}_0$ and $\mathcal{A}_{\mathcal{H}}$.

\begin{lem}
\label{lem:critial_points_of_functionals} $ $
\begin{enumerate}
  \item  Critical points of $\mathcal{A}_0$ are (up to reparametrisation) Reeb orbits and constant loops in $\Sigma$. More precisely, if $(x,r, \eta)$ is a critical point of $\mathcal{A}_0$ then $r(t) =1$ and $x(t) = \theta_{ \eta t}(x(0))$ is either a closed Reeb orbit (if $\eta >0$) or a closed Reeb orbit traversed backwards (if $\eta <0$) or a constant loop based at the point $x(0) \in \Sigma$ (if $\eta = 0$). In all three cases the critical value $\mathcal{A}_0(x,r,\eta) = \eta$, which is also the action of $\alpha$ on $x$. 
  \item Critical points of $\mathcal{A}_\mathcal{H}$  are translated points of $\varphi_1^{-1}$  in the sense of Sandon \cite{Sandon2013}. If  $(x,r, \eta)$ is a critical point of $\mathcal{A}_{\mathcal{H}}$ then $p := x(\tfrac12)$ is a translated point of $\varphi_1^{-1}$ with time-shift $-\eta$, that is \begin{equation}\label{eqn:translated_point}
   \theta_{-\eta}(p) = \varphi_1^{-1}(p), \qquad \bar{\rho}_1(p) = 1.
 \end{equation} 
 More precisely, up to reparametrisation $x(t)$ follows first the flow $\theta_{\eta t}$ from $x(0)$ to $p=x(\tfrac12)=\theta_\eta(x(0))$ and then the flow $\varphi_t^{-1}$ from $p$ to $x(1)=\varphi_t^{-1}(\theta_{ \eta }(x(0)))$. Thus, $x(1)=x(0)$ implies  \eqref{eqn:translated_point}.
 
Moreover $\mathcal{A}_{ \mathcal{H}}(x,r, \eta) =\eta$. In particular, whilst the function $\mathcal{H}$ depends on the choice of $\delta >0$ that occurs in the cutoff functions $\beta_\delta$ and $m_\delta$ in Definition \ref{defn:cutoff_functions}, the critical points and critical values of $\mathcal{A}_{\mathcal{H}}$ do not.
\end{enumerate}
\end{lem}
We will refer to critical points of $\mathcal{A}_0$ with $\eta = 0$ as \emph{constant critical points}.

\begin{proof}
The proof for $\mathcal{A}_0$ is well known, see for instance \cite{CieliebakFrauenfelder2009}. The proof for $\mathcal{A}_\mathcal{H}$ is also standard and uses the argument from \cite{AlbersFrauenfelder2010c} and \cite[Lemma 2.2]{AlbersMerry2013a} together with  the observation of Shelukhin explained in Remark \ref{rem:motivating_K}, cf. \cite[Lemma 42]{Shelukhin2014}. More specifically, if $(x,r,\eta)\in \mathrm{Crit}\mathcal{A}_\mathcal{H}$, then $z=(x,r)$ and $\eta$ satisfy \begin{equation}\label{eq:crit}\partial_tz=\eta X_{\chi m_\delta}(z)+X_{\mathcal{H}}(z) \quad \text{and} \quad \int_{S^1} \chi(t)m_\delta(r(t))\,dt=0\end{equation} 
by an analogue of (\ref{eq:gradient}) for $\mathcal{A}_\mathcal{H}$. Note that $X_{\chi m_\delta}$ vanishes for $t\in [\frac{1}{2},1]$ while $X_{\mathcal{H}}$ vanishes for $t\in[0,\frac{1}{2}]$, since the functions $\chi(t)$ resp. $\kappa(t)$ (and thus $\mathcal{H}_t$, see Definition \ref{defn:mathcalH}) have this property by Definition \ref{defn:cutoff_functions}. Therefore 
$z(t)$ follows the flow of $X_{\chi m_\delta}$ for $t\in [0,\frac{1}{2}]$ and is thus contained in a fixed level set of the Hamiltonian function $\chi m_\delta$. By the second equation in (\ref{eq:crit}) it follows that $r(t)=1$ and thus $z(t)\in \Sigma\times\{1\}\subset S\Sigma$; therefore $x(t)$ follows a reparametrization of the Reeb flow for $t\in[0,\frac{1}{2}]$ and $z(\frac{1}{2})=(\theta_\eta(x(0)),1)$. For $t\in [\frac{1}{2},1]$ we have $\partial_tz=X_\mathcal{H}(z)$ and thus $z(t)$ follows the flow $\Phi_t^\mathcal{H}$ of $X_\mathcal{H}$. The flow $\Phi_t^{\mathcal{H}}$ is  by Remark \ref{rem:motivating_K} along $\Sigma\times \{1\}$
just a reparametrization of $\phi_t^{-1}$ and thus we have $\Phi_1^\mathcal{H}(y,1)=\phi_1^{-1}(y,1)$ for all $(y,1)\in S\Sigma$.

 As a consequence $z(1)=\phi_1^{-1}(z(\frac{1}{2}))$ and hence $z(1)=z(0)=(x(0),1)$, since $z$ is a loop. The equality $\phi_1^{-1}(x(\frac{1}{2}),1)=(x(0),1)$ implies by Lemma \ref{lem:inverse} and equation \eqref{eq:phip} that $\bar{\rho}_1(x(\frac{1}{2}))=1$. Thus if we set $p:=x(\frac{1}{2})$ the pair $(p,\eta)$ satisfies  \eqref{eqn:translated_point}.
Moreover, using the equations \eqref{eq:crit} and the identity $(r\alpha)(X_K)=\omega(r\partial_r, X_K)=dK(r\partial_r)$ for any Hamiltonian function $K$, we find that $\eta$ equals the critical value $\mathcal{A}_\mathcal{H}(x,r,\eta)$. 

Conversely, any pair $(p,\eta)\in \Sigma\times \R$ satisfying \eqref{eqn:translated_point} determines a unique critical point $(z=(x,r),\eta)$ of $\mathcal{A}_\mathcal{H}$, by finding $z:[0,1]\rightarrow S\Sigma$ solving (\ref{eq:crit}) with $z(0):=(\theta_{-\eta}(p),1)$. The resulting path $z:[0,1]\rightarrow S\Sigma$ gives rise to a loop since $\bar{\rho}_1(p)=1$.  
\end{proof} 

\begin{defn}
\label{defn:SFT_type}
We fix a complex structure $J$ on $S \Sigma$ which is compatible with $- \omega$ and of SFT-type. This means that $\omega(J \cdot ,\cdot)$ defines a Riemannian metric on $S \Sigma$, that $d r \circ J = r \alpha$ and that $J$ is invariant under the Liouville flow $(x,r) \mapsto (x, e^t r)$ for $t \in \R$.

In the Liouville fillable case $J$ is assumed to be SFT-like on $\Sigma\times[1,\infty)\subset\widehat W$ and compatible with $-\omega$ on all of $\widehat W$.
\end{defn}

Using $J$ we define an $L^{2}$-inner
product $\left\llangle \cdot,\cdot\right\rrangle$
on $\Lambda \times\R$, cf. Remark \ref{rem:loopspace_notation}: for $(z,\eta)\in\Lambda \times\R$, vector fields $\zeta_1, \zeta_2$ along the loop $z$ and 
$b_1,b_2\in \R$, set 
\begin{equation}
\left\llangle  (\zeta_1,b_1),(\zeta_2,b_2)\right\rrangle:=\int_{S^1}\omega_{z(t)} \big( J_{z(t)}\zeta_1(t),\zeta_2(t) \big)\,dt+b_1b_2.
\end{equation}
We now define the gradient flow equation for our various functionals. For simplicity we give the definition only for $\mathcal{A}^{\sigma}_s$, the others are defined analogously. We denote by $\nabla\mathcal{A}^{\sigma}_s$ the gradient
of $\mathcal{A}^{\sigma}_s$ with respect to $\left\llangle  \cdot,\cdot \right\rrangle$. Explicitly, this is the integro-differential
operator 
\begin{equation}\label{eq:gradient}
\nabla \mathcal{A}^{\sigma}_s(z,\eta)=
\left( J_z \left(\partial_t z-\eta  X_{\chi m_\delta} (z)-X_{\mathcal{H}^{\sigma}_{s,t}}(z)\right),-\int_{S^1}\chi(t) m_\delta(r(t))\, dt\right),
\end{equation}
where as usual $z= (x,r)$. A (negative) gradient flow line of $ \mathcal{A}^{\sigma}_s$ is by definition a solution $w= (u,\eta) : \R \to \Lambda  \times \R$ of the equation
\begin{equation}
\label{eq:gradient_flow_line}
\partial_{s}w(s)+\nabla \mathcal{A}^{\sigma}_s(w(s))=0.
\end{equation}
Note that the  functionals $\mathcal{A}^{\sigma}_s$ and $\mathcal{A}^{\pm}_s$ give rise to $s$-dependent gradient flow equations, whereas $\mathcal{A}_{\mathcal{H}}$ and $\mathcal{A}_0$ do not.
The \emph{energy} of such a flow line is given by 
\begin{equation}
\E(u,\eta):=\int_{-\infty}^{\infty} \|\partial_sw\|^2 \,ds\;.
\end{equation}
If $w$ is a gradient flow line of $\mathcal{A}_0$ or $\mathcal{A}^{\sigma}_s$ with finite energy then
\begin{equation}\label{eq:asymptotics}
   \lim_{s \to \pm   \infty} w(s) \in \crit \mathcal{A}_0,
\end{equation}
see \eqref{eq:asymptotic_A_sigma}. This is a standard argument, see \cite{Salamon1999} for a proof. Here the result is identical for $\mathcal{A}_s^\sigma$, since each gradient flow line $w$ of $\mathcal{A}_s^\sigma$ coincides for $|s|\geq \sigma+1$ with a gradient flow line of $\mathcal{A}_0$. Similarly if $w$ is a gradient flow line of $\mathcal{A}_{\mathcal{H}}$ with finite energy then 
\begin{equation}
   \lim_{s \to \pm   \infty} w(s) \in \crit \mathcal{A}_{\mathcal{H}}.
\end{equation}
If $w$ is a gradient flow line of  $\mathcal{A}^+_s$ with finite energy then
\begin{equation}
   \lim_{s \to -   \infty} w(s) \in \crit \mathcal{A}_0, \qquad  \lim_{s \to +   \infty} w(s) \in \crit \mathcal{A}_{\mathcal{H}},
\end{equation}
and finally if $w$ is a gradient flow line of $\mathcal{A}^-_s$ with finite energy then
\begin{equation}
   \lim_{s \to -   \infty} w(s) \in \crit \mathcal{A}_{\mathcal{H}}, \qquad  \lim_{s \to +   \infty} w(s) \in \crit \mathcal{A}_0,
\end{equation}
see \eqref{eq:asymptotic_A_+} and \eqref{eq:asymptotic_A_-}. Figure \ref{pic:flowlines} gives a schematic description of various gradient flow lines.

\begin{figure}[ht] 
\def\svgwidth{40ex} 
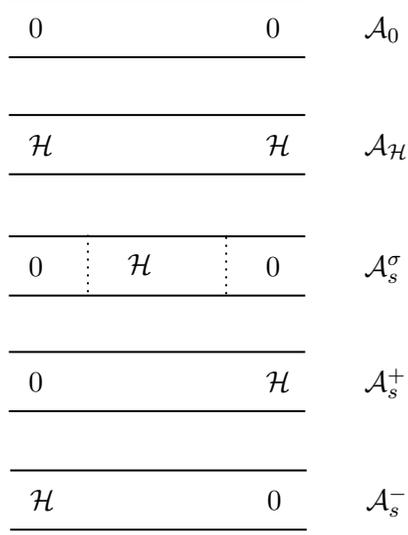
\caption{Flow lines.}\label{pic:flowlines}
\end{figure}
We now define the moduli space of gradient flow lines that will be needed in the proof of Theorem \ref{thm:technical_theorem_about_paths}. For this we first fix a constant critical point  $z_0 = (p,1,0)$ of $\mathcal{A}_0$.
\begin{defn}
\label{defn:M}
We denote by $\mathcal{M}$ the space of all pairs $(\sigma,w)$ where $w$ is a finite energy solution of \eqref{eq:gradient_flow_line} for $\mathcal{A}^{\sigma}_s$ such that 
\begin{equation}
\lim_{s \to  \infty} w(s)  = z_0=(p,1,0) \quad \text{and} \quad
\lim_{s \to  - \infty} w(s)  \text{ is a constant critical point of } \mathcal{A}_0.
\end{equation}
\end{defn} 

 We will need the following energy estimates.
\begin{lem}\label{lem:energy_of_things_in_M}
 Assume $h_t\geq 0$.
\begin{enumerate}
\item If $(\sigma, w) \in \mathcal{M}$ then 
\begin{equation}
  \E(w) \le \| \mathcal{H} \|_{\mathrm{osc}},
\end{equation}
and if $\sigma = 0$ then $\E(w) =0$. 
  \item Suppose that $w$ is a gradient flow line of $\mathcal{A}^+_s $ with finite energy. Set $z_1 := \lim_{s \to - \infty} w(s)$ and $z_2 := \lim_{s \to + \infty} w(s)$, and assume that $z_1$ is a constant critical point of $\mathcal{A}_0$.  Then
  \begin{equation}
    \E(w) \le - \mathcal{A}_{\mathcal{H}}(z_2)+\|\mathcal{H}\|_{\mathrm{osc}}.
  \end{equation}
  \item Suppose that $w$ is a gradient flow line of $\mathcal{A}^-_s $ with finite energy. Set $z_1 := \lim_{s \to - \infty} w(s)$ and $z_2 := \lim_{s \to + \infty} w(s)$, and assume that $z_2$ is a constant critical point of $\mathcal{A}_0$.  Then
  \begin{equation}
    \E(w) \le  \mathcal{A}_{\mathcal{H}}(z_1).
  \end{equation}
\end{enumerate}
\end{lem}
\begin{proof}
To prove (1), we denote by $z = \lim_{s \to -\infty}w(s)$ and  write $w = (u, \eta)$ and $u(s,t) = (x(s,t),r(s,t))$. Then we compute
by using (\ref{eq:gradient_flow_line})
\begin{equation} \begin{aligned}
\E(u ,\eta) & = \mathcal{A}_0(z) - \mathcal{A}_0(z_0) - \int_{- \infty}^{\infty}\int_0^1  \partial_s \mathcal{H}^{\sigma}_{s,t} (u(s,t))\, dt\,ds\\
& = \int_{- \infty}^{\infty}\int_0^1  -f_{\sigma}'(s)  \mathcal{H}_t(u(s,t))  \, dt\,ds\\
& = \int_{- \infty}^{0}\int_0^1  \underbrace{-f_{\sigma}'(s)}_{\le 0}  \mathcal{H}_t(u(s,t))  \, dt\,ds + \int_{0}^{\infty}\int_0^1  \underbrace{-f_{\sigma}'(s)}_{\ge 0} \mathcal{H}_t(u(s,t))  \, dt\,ds\\
& \le  \int_{- \infty}^{0} -f_{\sigma}'(s) \int_0^1   \min_{(x,r) } \mathcal{H}_t(x,r)  \, dt\,ds + \int_{0}^{\infty} -f_{\sigma}'(s) \int_0^1   \max_{(x,r)} \mathcal{H}_t(x,r)  \, dt\,ds\\
& = f_{\sigma}(0) \| \mathcal{H}\|_+-  f_{\sigma}(0) \| \mathcal{H} \|_- \\
& \le \| \mathcal{H} \|_{\mathrm{osc}}.
\end{aligned}\end{equation} 
If $\sigma = 0$ then since $f_{0}(0) = 0$ we obtain $\E(u ,\eta) = 0$.
To prove (2) we estimate
\begin{equation} \begin{aligned}
\E(u ,\eta) & = \mathcal{A}_0(z_1) - \mathcal{A}_\mathcal{H}(z_2) + \int_{- \infty}^{\infty}\int_0^1   \underbrace{-(f^+_{\infty})'(s)}_{\le 0} \mathcal{H}_t (u(s,t))\, dt\,ds\\
& \le - \mathcal{A}_\mathcal{H}(z_2) - \| \mathcal{H}\|_-.
\end{aligned}\end{equation} 
Similarly for (3) we obtain 
\begin{equation} \begin{aligned}
\E(u ,\eta) & = \mathcal{A}_\mathcal{H}(z_1) - \mathcal{A}_0(z_2) + \int_{- \infty}^{\infty}\int_0^1   \underbrace{-(f^-_{\infty})'(s)}_{\ge 0} \mathcal{H}_t (u(s,t))\, dt\,ds\\
& \le  \mathcal{A}_\mathcal{H}(z_1) + \| \mathcal{H}\|_+.
\end{aligned}\end{equation} 
Finally since $\mathcal{H}$ is non-positive and compactly supported by construction, we have $\| \mathcal{H} \|_+ = 0$ and thus $\| \mathcal{H} \|_{\mathrm{osc}} = -\| \mathcal{H}\|_-$.
 \end{proof} 
We denote by 
\begin{equation}
\label{eq:the_map_pi}
  \mathrm{pr} : \mathcal{M} \to [0, \infty), \qquad (\sigma, w) \mapsto \sigma
\end{equation}
the projection onto the first coordinate. Note that $\mathrm{pr}^{-1}(0)$ consists of a single point, namely $(0, z_0)$, where the constant critical point $z_0$ is thought of as a constant gradient flow line of $\mathcal{A}_0$, where we are using the last statement of part (1) of Lemma \ref{lem:energy_of_things_in_M}.

\begin{rem}
In the following Theorem we abbreviate by $C^{\infty}_{\mathrm{loc}}(\R, \Lambda(S \Sigma) \times \R)$ the space of pairs $w=(u,\eta)$ of smooth maps $\eta:\R\to\R$ and $u:\R\times S^1\to S\Sigma$ equipped with $C^{\infty}_{\mathrm{loc}}$-convergence. A key step in the proof of the main theorem is the following $C^{\infty}_{\mathrm{loc}}$-compactness statement.
\end{rem}

\begin{thm}
\label{thm:compactness} $ $
\begin{enumerate}
\item Assume that 
\begin{equation}
e^{5\delta}  \| h \|_+  < \varrho \cdot \wp( \nu_0),
\end{equation} 
where $\varrho=\varrho(h) >0$ was defined in \eqref{eq:inf_rho}. Then given any sequence $(\sigma_k, w_k) \in \mathcal{M}$ and any sequence $(s_k) \in \R$ of real numbers, the reparametrised sequence $w_k( \cdot +s_k)$ has  a subsequence which converges in $C^{\infty}_{\mathrm{loc}}(\R, \Lambda(S \Sigma) \times \R)$. 
\item Assume that $\Sigma$ is Liouville fillable. Then given any sequence $(\sigma_k, w_k) \in \mathcal{M}$ and any sequence $(s_k) \in \R$ of real numbers, the reparametrised sequence $w_k( \cdot +s_k)$ has  a subsequence which converges in $C^{\infty}_{\mathrm{loc}}(\R, \Lambda( \widehat W) \times \R)$. 
\end{enumerate}
\end{thm}
The proof of Theorem \ref{thm:compactness} uses ideas from SFT compactness, and is deferred to Section \ref{sec:SFTcompactness} below. The next result is the key technical argument needed to prove our main theorem. The argument is based on  ``neck stretching'' and similar to \cite[Section 2.1]{AlbersFrauenfelder2010c}.

\begin{thm}
\label{thm:technical_theorem_about_paths}$ $
\begin{enumerate} 
\item  Assume $h_t (x) \ge 0$  for all $x \in \Sigma$ and $t \in [0,1]$ is not identically $0$ and such that 
\begin{equation}
  \| h \|_+ < \varrho \cdot \wp(  \nu_0)
\end{equation}
where $\varrho=\varrho(h) >0$ was defined in \eqref{eq:inf_rho}. 
Then there exists a critical point $(x,r, \eta) $ of $\mathcal{A}_\mathcal{H}$ with 
\begin{equation}
  0 < \eta \le  \| h \|_+.
\end{equation}
\item Assume that $\Sigma$ is Liouville fillable and $h_t (x) \ge 0$  for all $x \in \Sigma$ and 
$t \in [0,1]$ is not identically $0$ and such that 
\begin{equation}
  \| h \|_+ <  \wp(  \nu_0)
\end{equation}
Then there exists a critical point $(x,r, \eta) $ of $\mathcal{A}_\mathcal{H}$ with 
\begin{equation}
  0 < \eta \le  \| h \|_+.
\end{equation}

\end{enumerate}
\end{thm}

\begin{proof}
We prove the statement in case (1). The proof in case (2) is simpler and explain the necessary adaption along the way.

We may assume that 
\begin{equation}\label{critical_values_bounded_away_from_0}
\inf\{\eta>0\mid\eta\text{ is a critical value of }\mathcal{A}_\mathcal{H}\}>0
\end{equation}
since otherwise there is nothing to prove. First choose $\delta_0 >0$ such that
 \begin{equation}
\label{eq:5delta}
  e^{5 \delta_0} \| h \|_+ < \varrho \cdot \wp(  \nu_0),
\end{equation}
where $\varrho=\varrho(h) >0$ was defined in \eqref{eq:inf_rho}. In case (2) we simply choose $\delta_0 >0$ such that  $e^{5 \delta_0} \| h \|_+ <  \wp(  \nu_0)$.

Fix $0 < \delta \le \delta_0$. We will show there exists a critical point $(x,r, \eta) $ of $\mathcal{A}_\mathcal{H}$ (where $\mathcal{H}$ is defined with this choice of $\delta$) with 
\begin{equation}
  0 < \eta \le e^{2 \delta} \| h \|_+.
\end{equation}
Since the critical points and critical values of $\mathcal{A}_{\mathcal{H}}$ do not depend on $\delta$ (cf.~the last line of Lemma \ref{lem:critial_points_of_functionals}), and $\delta$ can be arbitrarily small, the Theorem of Arzela-Ascoli together with \eqref{critical_values_bounded_away_from_0} implies that there is in fact a critical value $\eta$ satisfying $0 < \eta \le  \| h \|_+$.

Since $h\neq 0$ and $h_t(x)\geq0$ for all $x\in\Sigma$ and $t\in [0,1]$ we can choose $(p,t_0)\in\Sigma\times [0,1]$ with $h_{t_0}(p)>0$. Then we choose in Definition \ref{defn:M} for the moduli space $\mathcal{M}$ the critical point $z_0=(p,1,0)$. By assumption Theorem \ref{thm:compactness} is applicable. We prove the theorem in two steps.\\
\noindent \textbf{Step 1:} Assume that all critical points  $(x,r, \eta)$ of $\mathcal{A}_\mathcal{H}$ satisfy 
\begin{equation}
\label{eq:to_contradict}
   \eta \notin (0, e^{2 \delta}\| h \|_+].
\end{equation}
We will show that this implies the moduli space $\mathcal{M}$ is compact. In Step 2 below we will show this gives rise to a contradiction.\\

Let $(\sigma_k, w_k) $ be a sequence in $\mathcal{M}$. After passing to a subsequence, either $\sigma_k \to \sigma \in [0, \infty)$ or $\sigma_k \to \infty$. Assume the former. Then by Theorem \ref{thm:compactness}, after passing to another subsequence, the sequence $(w_k )$ converges to a (possibly) broken flow line of the form depicted in Figure \ref{pic:breaking1}, according to the statement in (\ref{eq:asymptotics}). In the next formula we use Lemma \ref{lem:Kosc_and_h} and the fact that $\varrho \le 1$, cf.~the line immediately after equation \eqref{eq:inf_rho}. If breaking occurs, then there exists a non-constant gradient flow line $w$ of $\mathcal{A}_0$ with energy
\begin{equation}
\label{eq:Ew}
\E (w) \le \limsup_k \E (w_k) \le \| \mathcal{H} \|_{\mathrm{osc}} \leq e^{2\delta}\|h\|_+\le \varrho \cdot e^{-3\delta} \wp(\nu_0)<\wp(\nu_0),
\end{equation}
and at least one of the asymptotic limits $\lim_{s \to \pm \infty}w(s)$ is a constant critical point. 

In case (2) we may apply Theorem \ref{thm:compactness} actually a smallness assumption of $h$. Our choice of $\delta_0$ guarantees that the estimate $\E (w) <\wp(\nu_0)$ corresponding to \eqref{eq:Ew} still holds.

Thus at least one of the asymptotic limits has action value zero. The energy $\E(w)$ of the non-constant gradient flow line $w$ is positive and is the difference of the action of the asymptotic limit as $\mathcal{A}_0$ does not depend on $s$. The other asymptote cannot be constant, since otherwise both ends would have action zero and thus the energy would be zero. Thus by part (1) of Lemma \ref{lem:critial_points_of_functionals}, the other end is a closed contractible Reeb orbit (possibly parametrised backwards), and hence the absolute value of the action value is at least $\wp( \nu_0)$. This gives $\E( w) \ge \wp( \nu_0)$, which contradicts \eqref{eq:Ew}. Thus in Figure \ref{pic:breaking1} no breaking occurs, and so $(\sigma_k, w_k)$ converges up to a subsequence to an element of $\mathcal{M}$. 

\begin{figure}[ht] 
\def\svgwidth{80ex} 
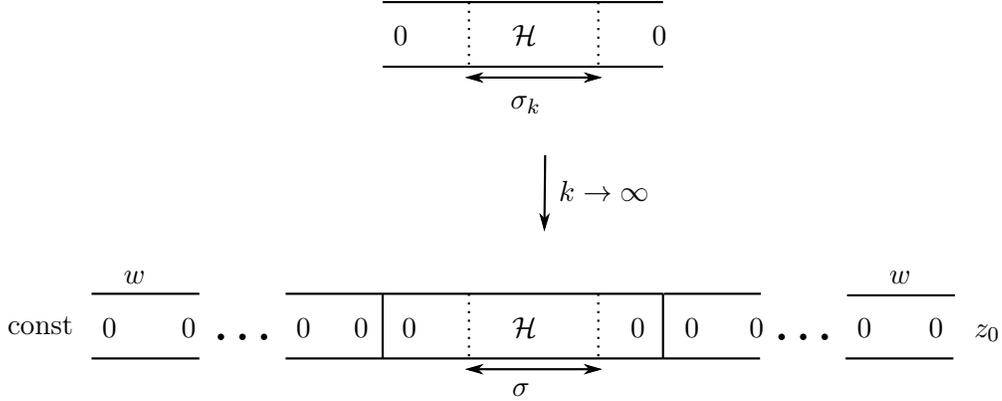
\caption{Breaking in the case $\sigma_k \to \sigma$.}\label{pic:breaking1}
\end{figure}

Now assume that $\sigma_k \to \infty$. Then by Theorem \ref{thm:compactness}, after passing to another subsequence, the sequence $(w_k)$ converges to a broken flow line, see Figure \ref{pic:breaking2}. 

\begin{figure}[ht] 
\def\svgwidth{85ex} 
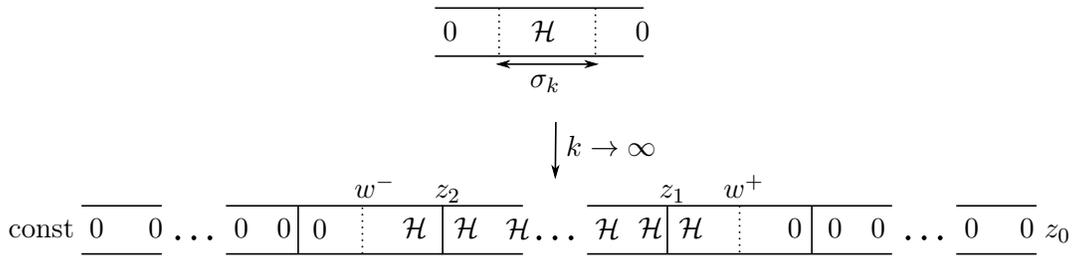
\caption{Breaking in the case $\sigma_k \to \infty$.}\label{pic:breaking2}
\end{figure} 

In fact, by arguing as above, there cannot be breaking into flow lines of $\mathcal{A}_0$.  Thus we obtain two gradient flow lines $w^{\pm}$ of the $s$-dependent functionals $\mathcal{A}^{ \mp}_s$ such that $(p,1,0)=z_0=\lim_{s \to  \infty} w^+(s)$  and $\lim_{s \to -\infty}w^-(s)$ are both constant critical points of $\mathcal{A}_0$. Set $z_1 := \lim_{s \to -\infty} w^+(s)$ and $z_2 := \lim_{s \to  \infty} w^-(s)$, see Figure \ref{pic:breaking3}. 

\begin{figure}[ht] 
\def\svgwidth{70ex} 
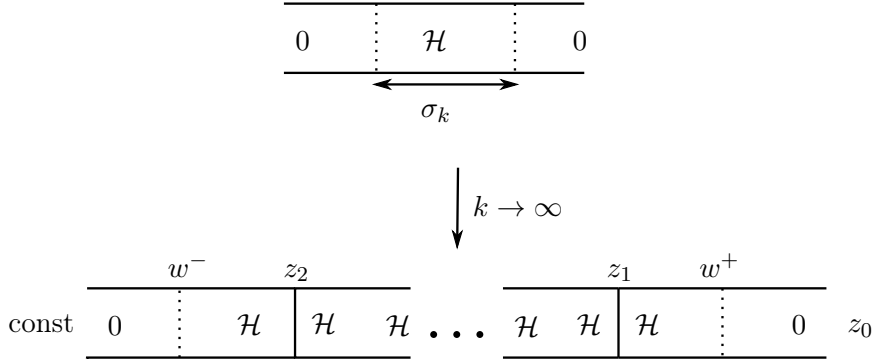
\caption{Breaking in the case $\sigma_k \to \infty$, (after simplification).}\label{pic:breaking3}
\end{figure} 
We first argue that $w^+$ is not constant since otherwise $z_0=z_1$ and therefore $z_0$ had to be a critical point of $\mathcal{A}_\mathcal{H}$. This is impossible due to our choice $z_0=(p,1,0)$ with $h_{t_0}(p)>0$ as we explain now. According to (2) in Lemma \ref{lem:critial_points_of_functionals} critical points of $\mathcal{A}_\mathcal{H}$ have to first follow the flow of $\theta_{\eta t}$. In our case $\eta$ equals $0$ and thus the critical point stays at $p$. After that the critical point has to follow the flow $\varphi_t^{-1}$. Since the critical point $z_0$ is constant the point $p$ needs to be fixed by $\varphi_t^{-1}$ for all $t$. This contradicts our choice of $p$, indeed since there exists $t_0$ with $h_{t_0}(p)>0$ the flow $\varphi_t$ (whose contact Hamiltonian is $h_t$) cannot fix $p$ for all $t$. Thus $w^+$ is not constant and therefore we get \begin{equation}
\E(w^+) >0.
\end{equation}
Moreover since $z_2$ is connected to $z_1$ by a sequence of gradient flow lines of $\mathcal{A}_\mathcal{H}$, we have 
\begin{equation}\label{eq:action_ineq}
  \mathcal{A}_\mathcal{H}(z_1) \le \mathcal{A}_\mathcal{H}(z_2).
\end{equation}
From Lemma \ref{lem:energy_of_things_in_M} parts (2) and (3), we have
\begin{equation}\label{eq:w_minus}
  0 \le \E(w^-) \le  - \mathcal{A}_\mathcal{H}(z_2) +\|\mathcal{H}\|_{\mathrm{osc}}
\end{equation}
and 
\begin{equation}\label{eq:w_plus}
   0 <\E(w^+) \le \mathcal{A}_\mathcal{H}(z_1).
\end{equation}
Combining the last three inequalities and Lemma \ref{lem:Kosc_and_h}, we obtain
\begin{equation}
  0 \stackrel{\eqref{eq:w_plus}}{<} \mathcal{A}_\mathcal{H}(z_1) \stackrel{\eqref{eq:action_ineq}}{\le} \mathcal{A}_\mathcal{H}(z_2) \stackrel{\eqref{eq:w_minus}}{\le}  \| \mathcal{H} \|_{\mathrm{osc}} \stackrel{\text{Lemma }\ref{lem:Kosc_and_h}}{\le} e^{2 \delta }\|h \|_+.
\end{equation}
This contradicts \eqref{eq:to_contradict}, and hence this latter case cannot occur. Thus $\mathcal{M}$ is compact as claimed. This finishes the proof of Step 1. \\

\noindent \textbf{Step 2:}
We now obtain a contradiction from the assertion that the moduli space $\mathcal{M}$ is compact. \\

Let us assume now that the moduli space $\mathcal{M}$ is compact. $\mathcal{M}$ is the zero-set of a Fredholm section of index $1$  of a Banach-space bundle over a Banach manifold.  In more detail we consider the Banach manifold $\mathcal{B}$ consisting of triples $(\sigma,u,\eta)$ with $\sigma\in[0,\infty)$ and $W^{1,p,\epsilon}$-maps, $p>2$, $\eta:\R\to\R$ and $u:\R\times S^1\to S\Sigma$. Here the superscript $\epsilon$ refers to required exponential decay of rate $\epsilon>0$ at the asymptotic ends. The value of $\epsilon>0$ is smaller than the spectral gap of the Hessian of $\mathcal{A}_0$ at its critical points, see for instance \cite[Appendix A]{Frauenfelder2004} or \cite[Section 5]{BourgeoisOancea2010}. On $\mathcal{B}$ we consider a Banach-bundle $\mathcal{E}\to\mathcal{B}$ whose fiber over $(\sigma,u,\eta)$ is given by $L^{p,\epsilon}(\R,\R)\times L^{p,\epsilon}(u^*TS\Sigma)$. Then equation \eqref{eq:gradient_flow_line} represents the zero-set of a Fredholm section $\mathcal{F}:\mathcal{B}\to\mathcal{E}$. The index of $\mathcal{F}$ is $1$ due to the parameter $\sigma$. Indeed, for fixed $\sigma$ we obtain a Fredholm operator of index $0$, as can be seen at the special zero $(0,z_0)$.

Recall the map $\mathrm{pr} : \mathcal{M} \to [0, \infty)$ from \eqref{eq:the_map_pi}, and  that $\mathrm{pr}^{-1}(0)$ consists of a single point $(0,z_0)$. The Fredholm section $\mathcal{F}$ is transverse at the  point $(0,z_0)$, since the functional $\mathcal{A}_0$ is Morse-Bott at the component $\Sigma \times \{1\} \times \{0\}$ of constant critical points, cf. \cite[Lemma 2.12]{AlbersFrauenfelder2010c}. Thus, this Fredholm section is also transverse in a small neighbourhood of $(0,z_0)$. Since its zero-set $\mathcal{M}=\{\mathcal{F}=0\}$ is compact, this Fredholm section can be perturbed to a transverse Fredholm section $\widetilde{\mathcal{F}}$ in such a way that $\mathcal{M}^{\mathrm{reg}}:=\{\widetilde{\mathcal{F}}=0\}$ remains compact. Since the section $\mathcal{F}$ is already transverse near $(0,z_0)$, it suffices to perturb $\mathcal{F}$ away from this point, see \cite[Theorems 5.5 and 5.13]{HoferWysockiZehnder2014} or \cite[Theorem 5.21]{HoferWysockiZehnder2009}. Since the index of $\widetilde{\mathcal{F}}$ is $1$ we conclude that $\mathcal{M}^{\mathrm{reg}}$ is a smooth compact one-dimensional manifold with boundary  $ \partial \mathcal{M}^{\mathrm{reg}}$ equal to the point $(0,z_0)$. This is a contradiction to the classification of 1-manifolds and finishes the proof of Theorem \ref{thm:technical_theorem_about_paths}.
\end{proof}

Here is the main result of the paper, which was stated as Theorem \ref{thm:thm1_intro} resp.~Theorem \ref{thm:thm1_intro_Liouville_fillable} in the Introduction. For clarity, in the statement we denote the positive loop by $\psi$ instead of $\varphi$, since in the proof we will modify $\psi$ to a positive path $\varphi$, which we will then apply Theorem \ref{thm:technical_theorem_about_paths} to. 
\begin{thm}
\label{thm:theorem_1_in_the_paper} $ $
\begin{enumerate} 
\item Suppose $\psi = \{ \psi_t \}_{ t \in S^1}$ is a loop of contactomorphisms with contact Hamiltonian $k_t$. Denote by $\sigma_t : \Sigma \to (0 , \infty)$ the conformal factor of $\psi$ (satisfying $\psi_t^* \alpha = \sigma_t \alpha$), and set 
\begin{equation}
\varrho := \min_{(x,t) \in \Sigma \times [0,1]} \sigma_t (x).
\end{equation} 
If
\begin{equation}
0< \| k \|_{\mathrm{osc}} < \varrho \cdot \wp(  \nu_0) 
\end{equation}
then there is a closed Reeb orbit belonging to the free homotopy class $\nu_{\psi}$ with period $\eta$ such that $\|k\|_-< \eta \leq \|k\|_+$.
As a consequence, if $\psi$ is a non-trivial non-negative loop with $\| k \|_{\mathrm{osc}} <\varrho\cdot  \wp(  \nu_0) $, then
\begin{equation}
  \wp(  \nu_{ \psi}) \le \| k \|_+.
\end{equation}
\item Suppose $\Sigma$ is Liouville fillable and $\psi = \{ \psi_t \}_{ t \in S^1}$ is a loop of contactomorphisms with contact Hamiltonian $k_t $. If 
\begin{equation} 0<\| k \|_{\mathrm{osc}} <  \wp(  \nu_0) 
\end{equation}
then there is a closed Reeb orbit belonging to the free homotopy class $\nu_{\psi}$ with period $\eta$ such that 
$\|k\|_-< \eta \leq \|k\|_+$.
As a consequence, if $\psi$ is a non-trivial non-negative loop with $\| k \|_{\mathrm{osc}} <  \wp(  \nu_0) $, then
\begin{equation}
  \wp(  \nu_{ \psi}) \le \| k \|_+.
\end{equation}
\end{enumerate}
\end{thm}

\begin{proof}
Again we prove the statement in case (1) and explain changes needed for case (2) along the way. Let us start with any
contact Hamiltonian $k_t:\Sigma\rightarrow \R$ and set 
$ \varpi_t:=\min_\Sigma k_t$ and $\varpi:=\int_0^1\varpi_s \,ds$. Then 
$\varpi=\|k\|_-$ and
\begin{equation}
k_t(x) \geq \varpi_t, \qquad \forall \,x \in \Sigma, \, t \in [0,1].
\end{equation}
Set $h_t (x):= k_t(x) - \varpi_t$ and observe that by construction $h_t\geq 0$, $\|h\|_-=0$ and $\|h\|_{\mathrm{osc}} =\|k\|_{\mathrm{osc}} $.
Actually $h_t$ is the contact Hamiltonian for the non-negative path 
\begin{equation}\label{eq:compos}
  \varphi_t := \theta_{ - \int_0^t\varpi_s \,ds} \circ \psi_t.
\end{equation}
 Moreover, if $\sigma_t:\Sigma\rightarrow (0,\infty)$ denotes the conformal factor of $\psi_t$
the conformal factor $\rho_t$ of $\varphi_t$ is given by 
\begin{equation}
  \rho_t(x) = \sigma_t \Big(\theta_{ - \int_0^t\varpi_s \,ds}(x) \Big),
\end{equation}
since the conformal factor of $\theta_t$ is identically 1. In particular,
\begin{equation}
\inf_{(x,t) \in \Sigma \times [0,1]} \sigma_t(x) = \inf_{(x,t) \in \Sigma \times [0,1]} \rho_t(x) \;.
\end{equation}
Consider first the special situation in which $\|k\|_{\mathrm{osc}}=0$: then $h_t=0$ and thus (\ref{eq:compos}) implies $\psi_t=\theta_{\int_0^t\varpi_s ds}$.
If $\psi$ is non-trivial non-negative loop of contactomorphisms, then the Reeb flow $\theta_t$ is periodic with period $\varpi=\|k\|_+>0$ and we obtain   $\wp(  \nu_{ \psi}) \le \| k \|_+$ in this situation.\\
From now on we will assume throughout that $\psi$ is a loop of contactomorphisms generated by a contact Hamiltonian $k_t$ with $0<\|k\|_{\mathrm{osc}} <\varrho\cdot  \wp(  \nu_0)$.\\
Then $h_t$, as constructed above, is non-negative, not identically $0$ and $\|h\|_+=\|h\|_{\mathrm{osc}} <\varrho \cdot \wp( \nu_0)$. In case (2) we similarly obtain $h_t$ with $\|h\|_+=\|h\|_{\mathrm{osc}} < \wp( \nu_0)$. Therefore in both cases
Theorem \ref{thm:technical_theorem_about_paths} is applicable to the path $\varphi_t$, and we deduce the existence of a critical point $(x,r, \eta)$ of the  functional $\mathcal{A}_{ \mathcal{H}}$ (associated to $h_t$) with 
\begin{equation}
\label{eq:eta_eq}
0 < \eta \le \| h\|_+.
\end{equation}
Set $p:= x(\tfrac12)$ and $q := x(0)$, which are both points in $\Sigma$. 
Since $\psi$ is a loop of contactomorphisms based at the identity, (\ref{eq:compos}) implies that $\varphi_1^{-1}=\theta_\varpi$. Combining this with part (2) of Lemma \ref{lem:critial_points_of_functionals} we have that
 \begin{equation}
  \theta_{ - \eta}(p) = \varphi_1^{-1}(p) = \theta_{\varpi}(p),
\end{equation}
and hence
\begin{equation}
  \theta_{\varpi + \eta }(p) = p.
\end{equation}
In other words, $p$ lies on a closed Reeb orbit $\gamma$ of period $\varpi+ \eta$, with $0<\eta<\|h\|_+$. \\
Since $\|k\|_-=\varpi$ and $\|h\|_+=\|k\|_{\mathrm{osc}} =\|k\|_+-\|k\|_-$, the period of the Reeb orbit $\gamma$ lies in the interval $(\|k\|_-,\|k\|_+]$. If $k_t$ is non-negative then $\wp( \nu_{\psi})\leq \|k\|_+$, since in this situation $\|k\|_-\geq 0$.\\
We claim that $\gamma$ belongs to the free homotopy class $\nu_{\psi}$. Up to time reparametrisation, the critical point $x : S^1 \to \Sigma$ satisfies
\begin{equation}
   x(t) = \begin{cases}
   \theta_{ 2 \eta t}(q), & t \in [0,\tfrac12], \\
   \psi_{2t-1}^{-1} \circ \theta_{ \int_0^{2t-1} \varpi_{s}ds}(p), & t \in [\tfrac12 ,1],
   \end{cases}
 \end{equation} 
 that is, $x(t)$ first follows the Reeb flow up to time $\eta$, and then flows along $\varphi_t^{-1} = \psi_t^{-1} \circ \theta_{ \int_0^t\varpi_s ds }$. See the first loop in Figure \ref{pic:figure8}. Since the composition $\psi_t^{-1} \circ \theta_{ \int_0^t\varpi_s ds}$ is homotopic to the concatenation of $\psi_t^{-1} $ and $ \theta_{ \int_0^t\varpi_s ds}$, the loop $x$ is freely homotopic to the figure-8 loop $y(t)$ defined by
 \begin{equation}
   y(t) = \begin{cases}
   \theta_{ (\eta + \varpi) 2t}(q), & t \in [0,\tfrac12], \\
   \psi_{2t-1}^{-1}( \theta_{ \eta + \varpi}(q)), & t \in [\tfrac12 ,1].
   \end{cases}
 \end{equation} 
 See the second and third loops in Figure \ref{pic:figure8}.
\begin{figure}[ht] 
\def\svgwidth{70ex} 
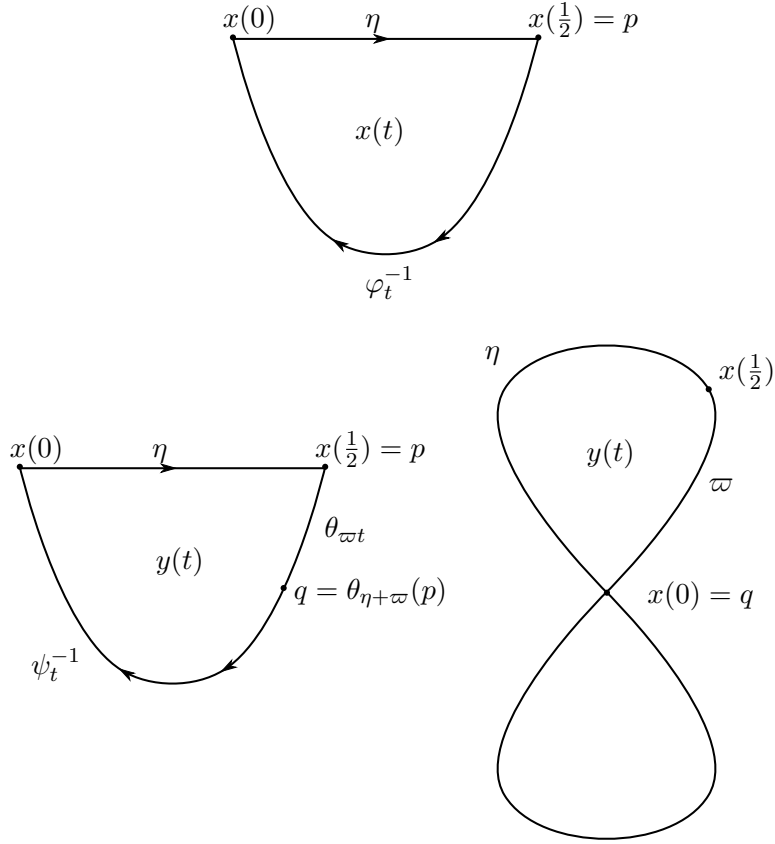
\caption{The loops $x(t)$ and $y(t)$.}\label{pic:figure8}
\end{figure}

Since the Reeb orbit $\gamma$ is the first half of $y$, and the loop $y$ itself is contractible, it follows that $\gamma$ belongs to $\nu_{\psi}$, as the second half of $y$ belongs to $\nu_{\psi^{-1}}$.
\end{proof}

\section{SFT compactness}
\label{sec:SFTcompactness}
In this section we prove the $C^{\infty}_{\mathrm{loc}}$-compactness result, stated as Theorem \ref{thm:compactness}. If $\Sigma$ is Liouville fillable then Theorem \ref{thm:compactness} has already been proved in \cite[Theorem 2.9]{AlbersFrauenfelder2010c}. The proof for the symplectisation case uses ideas from SFT compactness.

\begin{defn}
\label{defn:HoferEnergy} Let $(Z,j)$ denote a compact Riemann surface
(possibly disconnected and with boundary). 
 Fix an almost complex structure $J$ on $S \Sigma$ which is $-\omega$ compatible and of SFT-type (cf. Definition \ref{defn:SFT_type}). Suppose $u:Z\rightarrow S \Sigma $ is a $(j,J)$-holomorphic
map.  Write as usual $u=(x,r)$ so that $x : Z \to \Sigma$ and  $r:Z \to \R^+$.
Define $\tilde{u} : Z \to \Sigma \times \R$ by $\tilde{u}(z) = (x(z), \log r(z))$. We define the \textbf{Hofer energy} $\mathbf{E}(u)$ of $u$ as 
\begin{equation}
\mathbf{E}(u)=\sup_{\sigma\in\mathcal{S}}\int_{Z}\tilde{u}^*d(\sigma\alpha) \in [0,+\infty],
\end{equation}
where $\mathcal{S}:=\left\{ \sigma \in C^{\infty}(\R,[0,1])\mid\sigma'\geq0\right\} $.
\end{defn}

The following result appears explicitly in \cite[Theorem 5.3]{AlbersFuchsMerry2015}, but it is a minor variation on similar results in \cite{BourgeoisEliashbergHoferWysockiZehnder2003,CieliebakMohnke2005,Fish2011}. The only novelty is that no non-degeneracy assumptions on $\alpha$ are made. 
\begin{thm}
\label{thm:SFT} Let $J$ denote an almost complex structure on $S \Sigma$ which is compatible with $- \omega$ and of SFT-type.
Suppose $(Z_k,j_k)_{k \in \N}$ is a family of compact (possibly disconnected) Riemann surfaces with boundary and uniformly bounded genus. Assume that 
\begin{equation}
u_k=(x_k,r_k):Z_k\to S \Sigma
\end{equation}
is a sequence of $(j_k,J)$-holomorphic maps which have uniformly bounded Hofer energy $\mathbf{E}(u_k)\leq E$, are nonconstant on each connected
component of $Z_k$, and satisfy $r_k(\partial Z_k)\subset [e^a,+\infty)$ for some $a \in \R$. Asume that 
\begin{equation}
\inf_{k \in \N} \inf_{z \in Z_k} r_k(z)= 0.
\end{equation}
Then after passing to a subsequence, there exist cylinders $C_k \subset Z_k$, biholomorphically equivalent to
standard cylinders $[-L_k,L_k]\times S^1$, such that
$L_k\to + \infty$ and such that $u_k|_{C_k}$  converges (up to an $\R$-shift) in $C^\infty_{\emph{loc}}(\R\times S^1,S \Sigma)$ to a map  $u:\R\times S^1\to S \Sigma$  of the form $u(s,t)=(\gamma( \pm tT), e^{\pm Ts})$, where $\gamma$ is a closed orbit of $R$ of period $0 < T \le E$.
\end{thm}

\begin{proof}[Proof of Theorem \ref{thm:compactness}]
Case (2) is well-known, see for instance \cite[Theorem 2.9]{AlbersFrauenfelder2010c}. 

We will show that there exists $l>0$ such that the following holds. All elements $(u, \eta)$ with $(\sigma, u ,\eta) \in \mathcal{M}$ for some $\sigma \in \R^+$ satisfy that $u$ has image disjoint from $\Sigma \times (0, e^{-l})$. Then case (1) reduces to the argument in \cite[Theorem 2.9]{AlbersFrauenfelder2010c}. The existence of such $l$ was the main point addressed in \cite{AlbersFuchsMerry2015}, which we now explain. 

Assume for contradiction there exists a sequence of gradient flow lines $(\sigma_k, u_k, \eta_k) \in \mathcal{M}$  such that, writing $u_k = (x_k, r_k)$, one has  $\lim_{k \to +\infty} \inf r_k=0$. Recall  the definition of $\varrho$ from \eqref{eq:inf_rho}. Then, for a number $ 2 \delta  - \log \varrho < \varepsilon  \le 3  \delta - \log \varrho$ such that $e^{- \varepsilon}$ is a regular value
of all $r_k$'s, we set $Z_k := u_k^{-1}(\Sigma \times (0, e^{-\varepsilon}))$ and consider the map  $v_k : Z_k \to S \Sigma$ obtained by restricting $u_k$.  By \eqref{eq:supp}, the curves $v_k$ are $J$-holomorphic.  Since the gradient flow lines are asymptotic to critical points contained in $\Sigma \times \{1 \}$, each $Z_k$ is a compact
Riemann surface of genus $0$. 

Set $\tilde{v}_k = (x_k, \log r_k)|_{Z_k}$ as in Definition \ref{defn:HoferEnergy}. We have
\begin{equation}
\E(u_k, \eta_k) \geq \int_{Z_k}\tilde{v}_k^*d(e^s\alpha).
\end{equation}
On the other hand, by Stokes' Theorem we have for any $\sigma \in\mathcal{S}$ that
\begin{equation}
\int_{Z_k}\tilde{v}_k^*d(\sigma\alpha)\leq\int_{Z_k}\tilde{v}_k^*d\alpha=e^{\varepsilon}\int_{Z_k}\tilde{v}_k^*d(e^s\alpha),
\end{equation}
and thus  
\begin{equation}
\label{eq:hofer_energy_conversion}
  \mathbf{E}(v_k) \le e^{\varepsilon}\, \E(u_k, \eta_k).
  \end{equation}
Moreover the $J$-holomorphic curves $\tilde{v}_k$ have no constant components, and their Hofer energy is uniformly
bounded by 
\begin{equation}
e^{ \varepsilon} \E(u_k, \eta_k) \le e^{\varepsilon} \| \mathcal{H} \|_{\mathrm{osc}} ,
\end{equation} 
where we used \eqref{eq:hofer_energy_conversion} and part (1) of Lemma \ref{lem:energy_of_things_in_M}. Next, Lemma \ref{lem:Kosc_and_h} and the hypothesis of the theorem implies that
\begin{equation}
  \| \mathcal{H} \|_{\mathrm{osc}} \le e^{2 \delta} \| h \|_+<  \varrho \cdot e^{-3\delta}\wp( \nu_0),
\end{equation}
and hence we deduce that 
\begin{equation}
  \mathbf{E}(v_k) < e^{\varepsilon}\cdot \varrho \cdot e^{-3\delta} \wp(\nu_0) \le  \varepsilon^{  3 \delta - \log \varrho} \varrho \cdot e^{-3\delta}  \wp(\nu_0)=\wp( \nu_0).
\end{equation}
Thus by applying Theorem \ref{thm:SFT} to the pseudoholomorphic
curves $v_k$ it follows that there
exists a map $u_{k_0}$ (in fact a whole subsequence of the $u_k$ of such maps) with the following property: there is an embedded circle $S$ in the
domain $\R \times S^{1}$  of $u_{k_0}= (x_{k_0},r_{k_0})$,
such that the restriction of $x_{k_0}$ to $S$ parametrizes a circle
in $\Sigma$ homotopic to a Reeb orbit $\gamma$ of $\alpha$ which is of  period strictly less than
$\wp (\nu_0)$.
Since the domain of $u_{k_0}$ is $\R\times S^1$,
this circle $S$ bounds a disk $D$ in $\R \times S^1$,
or it is isotopic to a circle $\{s\}\times S^1\subset\R \times S^1$.
In either case $\gamma$ is contractible, since $x_{k_0}(S)$ is contractible. This is a contradiction to the definition of $\wp (\nu_0)$.
\end{proof}

\bibliographystyle{amsalpha}
\bibliography{willmacbibtex}
\end{document}